\pdfoutput=1
\documentclass[]{article}
\usepackage{amsmath,amssymb,amsthm}
\usepackage{enumerate}
\usepackage{graphicx}
\usepackage{enumitem}
\usepackage{pdfpages}
\usepackage[lmargin=3.9cm, rmargin=3.9cm, tmargin=3.2cm, bmargin=3.4cm,]{geometry}
\usepackage[margin=6ex,font=small]{caption}
\usepackage{lipsum}
\usepackage{tikz}
\usetikzlibrary{decorations.markings}

\title{The K-theory of the C*-algebras of $2$-rank graphs associated to complete bipartite graphs}
\author{S.A. Mutter \\
\textit{\normalsize School of Mathematics, Statistics and Physics, Newcastle University} \\
\texttt{s.a.mutter2@ncl.ac.uk}
}

\newcommand{\addresses}{{
		\bigskip
		\footnotesize
		\begin{center}
			\textsc{School of Mathematics, Statistics and Physics,\\ Newcastle University, Newcastle upon Tyne NE1 7RU, UK}\par\nopagebreak
			\texttt{s.a.mutter2@newcastle.ac.uk}
		\end{center}
}}

\theoremstyle{plain}
\newtheorem{thm}{Theorem}[section]
\newtheorem{prop}[thm]{Proposition}

\newtheorem{lem}[thm]{Lemma}
\newtheorem{corl}[thm]{Corollary}

\theoremstyle{definition}
\newtheorem{defn}[thm]{Definition}

\newtheorem{ex}[thm]{Example}

\newcommand{\inv}{^{-1}}
\DeclareMathOperator{\lcm}{lcm}
\DeclareMathOperator{\coker}{coker}
\DeclareMathOperator{\rk}{rk}
\DeclareMathOperator{\Ob}{Ob}
\DeclareMathOperator{\Hom}{Hom}
\DeclareMathOperator{\tors}{tor}
\newcommand{\cst}{C^{\star}}
\DeclareMathOperator{\ord}{ord}
\DeclareMathOperator{\im}{im}

\begin{document}

\maketitle

\begin{abstract}
	Using a result of Vdovina, we may associate to each complete connected bipartite graph $\kappa$ a $2$-dimensional square complex, which we call a tile complex, whose link at each vertex is $\kappa$. We regard the tile complex in two different ways, each having a different structure as a $2$-rank graph. To each $2$-rank graph is associated a universal $\cst$-algebra, for which we compute the K-theory, thus providing a new infinite collection of $2$-rank graph algebras with explicit K-groups. We determine the homology of the tile complexes, and give generalisations of the procedures to complexes and systems consisting of polygons with a higher number of sides.
\end{abstract}

\section{Introduction}

In \cite{Vdo2002}, it was shown how to construct a two-dimensional CW-complex whose link at each vertex is a complete bipartite graph. In \cite{KumPas2000}, generalising the work of \cite{RobSte1999}, certain combinatorial objects called \textit{higher-rank graphs} were defined and then associated a generalisation of a graph algebra \cite[Chapter 1]{Rae2005}. We combine these two methods to build an infinite family of $\cst$-algebras corresponding to complete bipartite graphs.

We begin in Section \ref{S_tile} by detailing Vdovina's construction of the CW-complexes, which we call \textit{tile complexes}; the data used to build these is called a \textit{tile system}. In Sections \ref{S_cst} and \ref{S_unpointed}, we associate adjacency matrices to the tile systems in two different ways: by considering the tiles as pointed, and as unpointed geometrical objects. By the fact that the adjacency matrices commute, they characterise the structure of a higher-rank graph, and as such induce a universal $\cst$-algebra, the \textit{higher-rank graph algebra}. We use a result of \cite{Eva2008} to calculate the K-groups of these algebras (Theorems \ref{pointed}, \ref{unpointed}).

In the brief Section \ref{S_homology}, we show that the tile complexes have torsion-free homology groups given by $H_1 \cong H_2 \cong \mathbb{Z}^{\alpha + \beta - 2}$, and $H_n = 0$ otherwise.

Finally, we explore extensions of these methods to $2t$\textit{-gon systems}, constructed analogously from two-dimensional complexes consisting entirely of $2t$-gons. In all, we associate $2$-rank graph $\cst$-algebras to five systems, and compute their K-theory in the following theorems:

\begin{enumerate}[label=(\roman*)]
	\item Pointed and unpointed tile systems (Theorems \ref{pointed}, \ref{unpointed}),
	\item Pointed and unpointed $2t$-gon systems, for even $t$ (Theorem \ref{pointed_polygon}, Corollary \ref{unpointed_polygon}),
	\item Pointed $2t$-gon systems, for arbitrary $t$ (Theorem \ref{new_polygons}),
\end{enumerate}
	
The respective systems in (ii) directly generalise those in (i), however there is another intuitive way of building $2t$-gon systems from polyhedra, (iii). We discuss the naturality of these generalisations in Section \ref{S_polygons}.

Our approach differs from that of Robertson and Steger, who focussed on complexes with one vertex. Furthermore, we use the terminology of higher-rank graphs in order to demonstrate the large intersection between the fields of $k$-graphs and geometry.

Throughout the paper, $\alpha$, $\beta$ are positive integers, and $\kappa(\alpha,\beta)$ denotes the complete connected bipartite graph on $\alpha$ white and $\beta$ black vertices.

\section{The tile system associated to a bipartite graph} \label{S_tile}

\begin{defn}\label{polyhedron}
	Let $t \in \mathbb{Z}$ with $t \geq 2$, and let $A_1,\ldots ,A_n$ be a sequence of solid $t$-gons, with directed edges labelled from some set $\mathcal{U}$. By gluing together like-labelled edges (respecting their direction), we obtain a two-dimensional complex $P$. We call such a complex a $t$\textbf{-polyhedron}.
	
	The \textbf{link} at a vertex $z$ of $P$ is the graph obtained as the intersection of $P$ with a small $2$-sphere centred at $z$.
\end{defn}

\begin{thm}[Vdovina, 2002]\label{vdovina}
	Let $G$ be a connected bipartite (undirected) graph on $\alpha$ white and $\beta$ black vertices, with edge set $E(G)$. Then we can construct a $2t$-polyhedron $P(G)$ which has $G$ as the link at each vertex, for each $t \geq 1$.
\end{thm}

We reference \cite{Vdo2002}, in which it was shown how to build such a $2t$-polyhedron. The general method is as follows:

Write $U'=\lbrace u_1 , \ldots , u_\alpha\rbrace$ for the set of white vertices of $G$, and $V'=\lbrace v_1,\ldots , v_\beta\rbrace$ for the set of black vertices.

Let $U$ be a set with $2t\alpha$ elements, indexed $u_i^1, u_i^2, \ldots , u_i^t, \bar{u}_i^1, \bar{u}_i^2,\ldots \bar{u}_i^t$ for each $u_i \in U'$, and let $V$ be the corresponding set with $2t\beta$ elements. Define fixed-point-free involutions $u_i^r \mapsto \bar{u}_i^r$ and $v_i^r \mapsto \bar{v}_i^r$ in $U$ and $V$, respectively.

Each edge of the graph $G$ joins an element of $U'$ to an element of $V'$; for each edge $e = u_p v_q$, we construct a $2t$-gon $A_e$ with a distinguished base vertex. Label the boundary of $A_e$ anticlockwise, starting from the base, by the sequence $u_p^1,v_q^1,u_p^2,v_q^2,\ldots , u_p^t,v_q^t$, giving each side of the boundary a forward-directed arrow. We denote this pointed oriented $2t$-gon by $A_e = \big[ u_p^1,v_q^1,\ldots , u_p^t,v_q^t \big]$.  Then, glue the $A_e$ together in the manner of Definition \ref{polyhedron} in order to obtain a $2t$-polyhedron $P(G)$ (Figure \ref{fig_polyhedron}).

\begin{figure}[h]
	\begin{center}
		\begin{tikzpicture}
		\draw[->,thick] (0,0) -- (0.75,0);
		\draw[thick] (0.75,0) -- (1.5,0);
		\draw[->,thick] (1.5,0) -- (1.5,0.75);
		\draw[thick] (1.5,0.75) -- (1.5,1.5);
		\draw[thick] (0,0) -- (0,0.75);
		\draw[<-,thick] (0,0.75) -- (0,1.5);
		\draw[thick] (0,1.5) -- (0.75,1.5);
		\draw[<-,thick] (0.75,1.5) -- (1.5,1.5);
		
		\draw[] (0.75,0.75) node{$A_{u_1 v_2}$};
		\draw[anchor=north] (0.75,-0.1) node{$u_1^1$};
		\draw[anchor=west] (1.6,0.75) node{$v_2^1$};
		\draw[anchor=south] (0.75,1.6) node{$u_1^2$};
		\draw[anchor=east] (-0.1,0.75) node{$v_2^2$};
		
		\filldraw (0,0) circle (2pt);
		
		\begin{scope}[xshift=3.5cm]
		\draw[->,thick] (0,0) -- (0.75,0);
		\draw[thick] (0.75,0) -- (1.5,0);
		\draw[->,thick] (1.5,0) -- (1.5,0.75);
		\draw[thick] (1.5,0.75) -- (1.5,1.5);
		\draw[thick] (0,0) -- (0,0.75);
		\draw[<-,thick] (0,0.75) -- (0,1.5);
		\draw[thick] (0,1.5) -- (0.75,1.5);
		\draw[<-,thick] (0.75,1.5) -- (1.5,1.5);
		
		\draw[] (0.75,0.75) node{$A_{u_1 v_3}$};
		\draw[anchor=north] (0.75,-0.1) node{$u_1^1$};
		\draw[anchor=west] (1.6,0.75) node{$v_3^1$};
		\draw[anchor=south] (0.75,1.6) node{$u_1^2$};
		\draw[anchor=east] (-0.1,0.75) node{$v_3^2$};
		
		\filldraw (0,0) circle (2pt);
		\end{scope}
		
		\begin{scope}[xshift=7cm]
		\draw[->,thick] (0,0) -- (0.75,0);
		\draw[thick] (0.75,0) -- (1.5,0);
		\draw[->,thick] (1.5,0) -- (1.5,0.75);
		\draw[thick] (1.5,0.75) -- (1.5,1.5);
		\draw[thick] (0,0) -- (0,0.75);
		\draw[<-,thick] (0,0.75) -- (0,1.5);
		\draw[thick] (0,1.5) -- (0.75,1.5);
		\draw[<-,thick] (0.75,1.5) -- (1.5,1.5);
		
		\draw[] (0.75,0.75) node{$A_{u_2 v_1}$};
		\draw[anchor=north] (0.75,-0.1) node{$u_2^1$};
		\draw[anchor=west] (1.6,0.75) node{$v_1^1$};
		\draw[anchor=south] (0.75,1.6) node{$u_2^2$};
		\draw[anchor=east] (-0.1,0.75) node{$v_1^2$};
		
		\filldraw (0,0) circle (2pt);
		\end{scope}
		\end{tikzpicture}
	\end{center}
	\caption{Construction of a $2t$-polyhedron: Give each side of a sequence of solid $2t$-gons a direction and a label from one of two sets $U$, $V$, then glue together corresponding sides with respect to their direction.}\label{fig_polyhedron}
\end{figure}
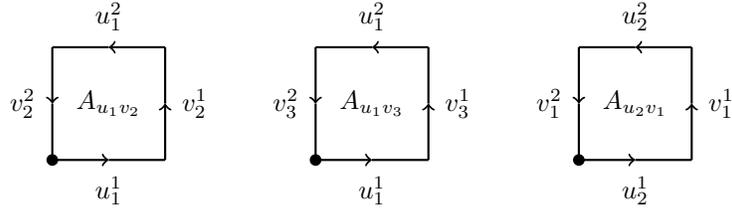

\begin{defn}\label{def_tiles}
	In this paper, we mainly concern ourselves with $4$-polyhedra, that is, those constructed by gluing together squares. We will refer to $4$-polyhedra as \textbf{tile complexes}. For a connected bipartite graph $G$, write $TC(G)$ for the tile complex $P(G)$, and define the set
	\begin{multline}\label{1}
	\mathcal{S}(G):= \big\lbrace A_e = \big[ u_p^1, v_q^1, u_p^2, v_q^2 \big], \big[ \bar{u}_p^1, \bar{v}_q^2, \bar{u}_p^2, \bar{v}_q^1 \big], \\
	\big[ u_p^2,v_q^2,u_p^1,v_q^1 \big], \big[ \bar{u}_p^2, \bar{v}_q^1, \bar{u}_p^1, \bar{v}_q^2 \big] \bigm\vert e = u_p v_q \in E(G) \big\rbrace .
	\end{multline}

	We call elements of $\mathcal{S}(G)$ \textbf{pointed tiles}. We define an equivalence relation which, for each $A_e$, identifies the four corresponding pointed tiles in (\ref{1}). We denote by $\mathcal{S}'(G)$ the quotient of $\mathcal{S}(G)$ by this relation, and we use round brackets, writing $A_e' = \big( u_p^1,v_q^1,u_p^2,v_q^2 \big)$ for the equivalence class of $A_e$ in $\mathcal{S}'(G)$. Then $\mathcal{S}'(G)$ is the set of geometric squares (that is, disregarding basepoint and orientation) of which $TC(G)$ consists. We call elements of $\mathcal{S}'(G)$ \textbf{unpointed tiles}.

	Notice that by placing the basepoint at the bottom-left vertex, we can arrange that the horizontal sides of each pointed tile be labelled by elements of $U$, and the vertical sides by elements of $V$, such that $\mathcal{S}(G) \subseteq U \times V \times U \times V$. Indeed, the four tuples in (\ref{1}) correspond to the four symmetries of a pointed tile which preserve this property (Figure \ref{fig_tiles}).

	Note also that by design, any two pointed tiles in $\mathcal{S}(G)$ are distinct, and any two adjacent sides of a tile uniquely determine the remaining two sides.
\end{defn}

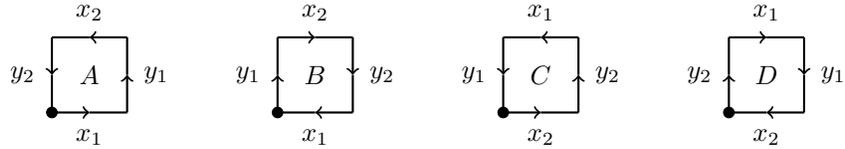
\begin{figure}[h]
	\begin{center}
		\begin{tikzpicture}
		\draw[->,thick] (0,0) -- (0.5,0);
		\draw[thick] (0.5,0) -- (1,0);
		\draw[->,thick] (1,0) -- (1,0.5);
		\draw[thick] (1,0.5) -- (1,1);
		\draw[thick] (0,0) -- (0,0.5);
		\draw[<-,thick] (0,0.5) -- (0,1);
		\draw[thick] (0,1) -- (0.5,1);
		\draw[<-,thick] (0.5,1) -- (1,1);
		
		\draw[] (0.5,0.5) node{$A$};
		\draw[anchor=north] (0.5,-0.1) node{$x_1$};
		\draw[anchor=west] (1.1,0.5) node{$y_1$};
		\draw[anchor=south] (0.5,1.1) node{$x_2$};
		\draw[anchor=east] (-0.1,0.5) node{$y_2$};
		
		\filldraw (0,0) circle (2pt);
		
		\draw[thick,shift={(3,0)}] (0,0) -- (0.5,0);
		\draw[<-,thick,shift={(3,0)}] (0.5,0) -- (1,0);
		\draw[thick,shift={(3,0)}] (1,0) -- (1,0.5);
		\draw[<-,thick,shift={(3,0)}] (1,0.5) -- (1,1);
		\draw[->,thick,shift={(3,0)}] (0,0) -- (0,0.5);
		\draw[thick,shift={(3,0)}] (0,0.5) -- (0,1);
		\draw[->,thick,shift={(3,0)}] (0,1) -- (0.5,1);
		\draw[thick,shift={(3,0)}] (0.5,1) -- (1,1);
		
		\draw[shift={(3,0)}] (0.5,0.5) node{$B$};
		\draw[anchor=north,shift={(3,0)}] (0.5,-0.1) node{$x_1$};
		\draw[anchor=west,shift={(3,0)}] (1.1,0.5) node{$y_2$};
		\draw[anchor=south,shift={(3,0)}] (0.5,1.1) node{$x_2$};
		\draw[anchor=east,shift={(3,0)}] (-0.1,0.5) node{$y_1$};
		
		\filldraw[shift={(3,0)}] (0,0) circle (2pt);
		
		\draw[->,thick,shift={(6,0)}] (0,0) -- (0.5,0);
		\draw[thick,shift={(6,0)}] (0.5,0) -- (1,0);
		\draw[->,thick,shift={(6,0)}] (1,0) -- (1,0.5);
		\draw[thick,shift={(6,0)}] (1,0.5) -- (1,1);
		\draw[thick,shift={(6,0)}] (0,0) -- (0,0.5);
		\draw[<-,thick,shift={(6,0)}] (0,0.5) -- (0,1);
		\draw[thick,shift={(6,0)}] (0,1) -- (0.5,1);
		\draw[<-,thick,shift={(6,0)}] (0.5,1) -- (1,1);
		\draw[shift={(6,0)}] (0.5,0.5) node{$C$};
		\draw[anchor=north,shift={(6,0)}] (0.5,-0.1) node{$x_2$};
		\draw[anchor=west,shift={(6,0)}] (1.1,0.5) node{$y_2$};
		\draw[anchor=south,shift={(6,0)}] (0.5,1.1) node{$x_1$};
		\draw[anchor=east,shift={(6,0)}] (-0.1,0.5) node{$y_1$};
		
		\filldraw[shift={(6,0)}] (0,0) circle (2pt);
		
		\draw[thick,shift={(9,0)}] (0,0) -- (0.5,0);
		\draw[<-,thick,shift={(9,0)}] (0.5,0) -- (1,0);
		\draw[thick,shift={(9,0)}] (1,0) -- (1,0.5);
		\draw[<-,thick,shift={(9,0)}] (1,0.5) -- (1,1);
		\draw[->,thick,shift={(9,0)}] (0,0) -- (0,0.5);
		\draw[thick,shift={(9,0)}] (0,0.5) -- (0,1);
		\draw[->,thick,shift={(9,0)}] (0,1) -- (0.5,1);
		\draw[thick,shift={(9,0)}] (0.5,1) -- (1,1);
		\draw[shift={(9,0)}] (0.5,0.5) node{$D$};
		\draw[anchor=north,shift={(9,0)}] (0.5,-0.1) node{$x_2$};
		\draw[anchor=west,shift={(9,0)}] (1.1,0.5) node{$y_1$};
		\draw[anchor=south,shift={(9,0)}] (0.5,1.1) node{$x_1$};
		\draw[anchor=east,shift={(9,0)}] (-0.1,0.5) node{$y_2$};
		
		\filldraw[shift={(9,0)}] (0,0) circle (2pt);
		\end{tikzpicture}
	\end{center}
	\caption{Visualisation of tiles: $A=[x_1,y_1,x_2,y_2]$, $B=[\bar{x}_1,\bar{y}_2,\bar{x}_2,\bar{y}_1]$, etc. These four pointed squares represent different pointed tiles, but the same unpointed tile.}\label{fig_tiles}
\end{figure}

\begin{defn}\label{tile_system}
	Let $G$ be a connected bipartite graph on $\alpha$ white and $\beta$ black vertices. Let $U$, $V$ be sets with $|U| = 4\alpha$, $|V| = 4\beta$, as constructed above, and let $\mathcal{S} = \mathcal{S}(G) \subseteq U \times V \times U \times V$ be the corresponding set of pointed tiles. We call the datum $(G, U, V, \mathcal{S})$ a \textbf{tile system}.
\end{defn}

This construction is closely related to, and indeed modelled on, that of a \textit{VH-datum}, introduced in \cite{Wis1996} and developed further in \cite{BurMoz2000}.

\section{The $\cst$-algebra corresponding to a tile system} \label{S_cst}

\begin{defn}\label{adjacency_matrices}
	Let $(G, U, V, \mathcal{S})$ be a tile system, and let $A=[x_1,y_1,x_2,y_2]$ and $B=[x_3,y_3,x_4,y_4]$ be pointed tiles in $\mathcal{S}$. We define two $4\alpha\beta \times 4\alpha\beta$ matrices $M_1,M_2$ with $AB$-th entry $M_i(A,B)$ as follows:
	\[
	\begin{array}{r}
	M_1(A,B) = \left\{ \begin{array}{l} 
	1 \quad \text{if $y_1 = \bar{y}_4$ and $x_1 \neq \bar{x}_3$,} \\
	0 \quad \text{otherwise,} 
	\end{array}
	\right.\\
	\text{}\\
	M_2(A,B) = \left\{ \begin{array}{l} 
	1 \quad \text{if $x_2 = \bar{x}_3$ and $y_1 \neq \bar{y}_3$,} \\
	0 \quad \text{otherwise,} 
	\end{array}
	\right.
	\end{array}
	\]
	as demonstrated in Figure \ref{fig_adjacency}. We call $M_1$ the \textbf{horizontal adjacency matrix} and $M_2$ the \textbf{vertical adjacency matrix}. If $M_i(A,B) = 1$, we say that $B$ is \textbf{horizontally-} or \textbf{vertically-adjacent} to $A$, respectively.
\end{defn}

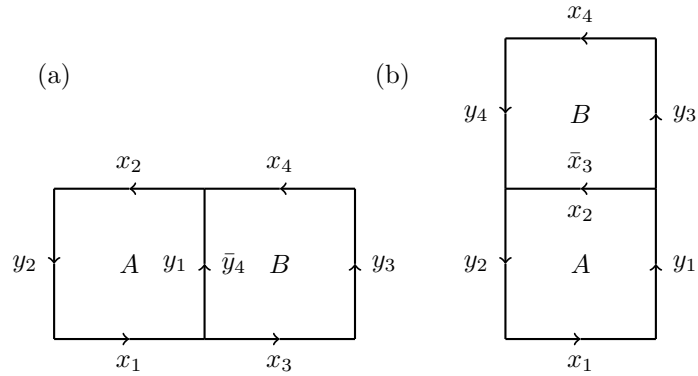
\begin{figure}[h]
	\begin{center}
		\begin{tikzpicture}
		\draw[->,thick] (0,0) -- (1,0);
		\draw[thick] (1,0) -- (2,0);
		\draw[->,thick] (2,0) -- (2,1);
		\draw[thick] (2,1) -- (2,2);
		\draw[thick] (0,0) -- (0,1);
		\draw[<-,thick] (0,1) -- (0,2);
		\draw[thick] (0,2) -- (1,2);
		\draw[<-,thick] (1,2) -- (2,2);
		\draw[->,thick] (2,0) -- (3,0);
		\draw[thick] (3,0) -- (4,0);
		\draw[->,thick] (4,0) -- (4,1);
		\draw[thick] (4,1) -- (4,2);
		\draw[thick] (2,2) -- (3,2);
		\draw[<-,thick] (3,2) -- (4,2);
		\draw[anchor=north] (1,-0.1) node{$x_1$};
		\draw[anchor=west] (2.1,1) node{$\bar{y}_4$};
		\draw[anchor=south] (1,2.1) node{$x_2$};
		\draw[anchor=east] (-0.1,1) node{$y_2$};
		\draw[anchor=east] (1.9,1) node{$y_1$};
		\draw[anchor=north] (3,-0.1) node{$x_3$};
		\draw[anchor=west] (4.1,1) node{$y_3$};
		\draw[anchor=south] (3,2.1) node{$x_4$};
		
		\draw[->,thick,shift={(6,0)}] (0,0) -- (1,0);
		\draw[thick,shift={(6,0)}] (1,0) -- (2,0);
		\draw[->,thick,shift={(6,0)}] (2,0) -- (2,1);
		\draw[thick,shift={(6,0)}] (2,1) -- (2,2);
		\draw[thick,shift={(6,0)}] (0,0) -- (0,1);
		\draw[<-,thick,shift={(6,0)}] (0,1) -- (0,2);
		\draw[thick,shift={(6,0)}] (0,2) -- (1,2);
		\draw[<-,thick,shift={(6,0)}] (1,2) -- (2,2);
		\draw[->,thick,shift={(6,0)}] (2,2) -- (2,3);
		\draw[thick,shift={(6,0)}] (2,3) -- (2,4);
		\draw[thick,shift={(6,0)}] (0,2) -- (0,3);
		\draw[<-,thick,shift={(6,0)}] (0,3) -- (0,4);
		\draw[thick,shift={(6,0)}] (0,4) -- (1,4);
		\draw[<-,thick,shift={(6,0)}] (1,4) -- (2,4);
		\draw[anchor=north,shift={(6,0)}] (1,-0.1) node{$x_1$};
		\draw[anchor=west,shift={(6,0)}] (2.1,1) node{$y_1$};
		\draw[anchor=south,shift={(6,0)}] (1,2.1) node{$\bar{x}_3$};
		\draw[anchor=east,shift={(6,0)}] (-0.1,1) node{$y_2$};			
		\draw[anchor=north,shift={(6,0)}] (1,1.9) node{$x_2$};
		\draw[anchor=east,shift={(6,0)}] (-0.1,3) node{$y_4$};
		\draw[anchor=west,shift={(6,0)}] (2.1,3) node{$y_3$};
		\draw[anchor=south,shift={(6,0)}] (1,4.1) node{$x_4$};
		
		\draw[] (1,1) node{$A$};
		\draw[] (3,1) node{$B$};
		\draw[] (7,1) node{$A$};
		\draw[] (7,3) node{$B$};
		
		\draw[] (0,3.5) node{(a)};
		\draw[] (4.5,3.5) node{(b)};
		
		\end{tikzpicture}
	\end{center}
	\caption{Horizontal and vertical adjacency: (a) $M_1(A,B) = 1$, (b) $M_2(A,B) = 1$.}\label{fig_adjacency}
\end{figure}

\begin{defn}\label{uce}
	Let $(G, U, V, \mathcal{S})$ be a tile system, and let $A$, $B$, $C$ be pointed tiles in $\mathcal{S}(G)$ such that $M_1(A,B) = 1$ and $M_2(A,C) = 1$. We say that the tile system $(G, U, V, \mathcal{S})$ satisfies the \textbf{Unique Common Extension Property} (\textbf{UCE property}) if there exists a \emph{unique} $D \in \mathcal{S}$ such that $M_2(B,D) = M_1(C,D) = 1$ (Figure \ref{fig_uce}).
\end{defn}

\begin{figure}[h]
	\begin{center}
		\begin{tikzpicture}
		\draw[->,thick] (0,0) -- (1,0);
		\draw[thick] (1,0) -- (2,0);
		\draw[->,thick] (2,0) -- (2,1);
		\draw[thick] (2,1) -- (2,2);
		\draw[->,thick] (0,0) -- (0,1);
		\draw[thick] (0,1) -- (0,2);
		\draw[->,thick] (0,2) -- (1,2);
		\draw[thick] (1,2) -- (2,2);
		
		\draw[->,thick] (2,2) -- (2,3);
		\draw[thick] (2,3) -- (2,4);
		\draw[->,thick] (0,2) -- (0,3);
		\draw[thick] (0,3) -- (0,4);
		\draw[->,thick] (0,4) -- (1,4);
		\draw[thick] (1,4) -- (2,4);
		
		\draw[->,thick] (2,0) -- (3,0);
		\draw[thick] (3,0) -- (4,0);
		\draw[->,thick] (4,0) -- (4,1);
		\draw[thick] (4,1) -- (4,2);
		\draw[->,thick] (2,2) -- (3,2);
		\draw[thick] (3,2) -- (4,2);
		
		\draw[->,thick,dashed] (4,2) -- (4,3);
		\draw[thick,dashed] (4,3) -- (4,4);
		\draw[->,thick,dashed] (2,4) -- (3,4);
		\draw[thick,dashed] (3,4) -- (4,4);
		
		\draw[] (1,1) node{$A$};
		\draw[] (3,1) node{$B$};
		\draw[] (1,3) node{$C$};
		\draw[] (3,3) node{$D$};
		\end{tikzpicture}
	\end{center}
	\caption{Unique Common Extension Property: Given an initial tile $A$, a horizontally-adjacent tile $B$, and a vertically-adjacent tile $C$, there is a unique tile $D$ adjacent to both $B$ and $C$.}\label{fig_uce}
\end{figure}
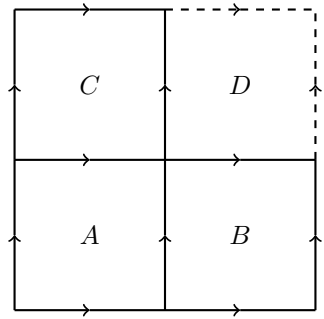

\begin{prop}\label{tile_complex_uce}
	Consider the complete bipartite graph $\kappa = \kappa(\alpha,\beta)$ on $\alpha \geq 2$ white and $\beta \geq 2$ black vertices, and let $(\kappa, U, V, \mathcal{S}(\kappa))$ be a tile system with corresponding adjacency matrices $M_1$, $M_2$. Then:
		\begin{enumerate}[label=(\roman*)]
			\item $M_1$ and $M_2$ are symmetric and commute with each other,
			\item Each row and column of $M_1$ and $M_2$ contains at least one non-zero element,
			\item $(\kappa, U, V, \mathcal{S}(\kappa))$ satisfies the UCE Property.
		\end{enumerate}
\end{prop}

\begin{proof}
	It is straightforward to verify that the matrices $M_1$ and $M_2$ are symmetric. Now, consider the pointed tile $A = \big[ u_i^1,v_j^1,u_i^2,v_j^2 \big] \in \mathcal{S}(\kappa)$ (any other tiles may be dealt with in a similar manner), and define the sets
	\[
	X_A := \lbrace T \in \mathcal{S}(\kappa) \mid M_1(A,T) = 1\rbrace, \quad Y_A := \lbrace T \in \mathcal{S}(\kappa) \mid M_2(A,T) = 1\rbrace.
	\]
	So $X_A$ comprises precisely those tiles of the form $\big[ \bar{u}_k^1, \bar{v}_j^2, \bar{u}_k^2, \bar{v}_j^1\big]$, where $k \neq i$, and $Y_A$ only those of the form $\big[ \bar{u}_i^2, \bar{v}_l^1, \bar{u}_i^1, \bar{v}_l^2\big]$, where $l \neq j$. Since $\alpha,\beta \geq 2$, these sets are non-empty.
	
	Now we can define non-empty sets $(YX)_A := \bigcup_{T \in X_A}Y_T$ and $(XY)_A := \bigcup_{T \in Y_A} X_T$. Notice that $(YX)_A = (XY)_A = \big\lbrace \big[ u_k^2, v_l^2, u_k^1, v_l^1\big] \bigm\vert k \neq i \text{ and } l \neq j \big\rbrace$ for all $A$ in $\mathcal{S}(\kappa)$, and therefore that $M_2 M_1 = M_1 M_2$.
	
	To show (iii), choose elements $B=\big[ \bar{u}_{k_0}^1, \bar{v}_j^2, \bar{u}_{k_0}^2, \bar{v}_j^1\big] \in X_A$, and $C=\big[  \bar{u}_i^2, \bar{v}_{l_0}^1, \bar{u}_i^1, \bar{v}_{l_0}^2\big] \in Y_A$. Then $D = \big[ u_{k_0}^2, v_{l_0}^2, u_{k_0}^1, v_{l_0}^1\big]$ is the unique pointed tile adjacent to both $B$ and $C$.
\end{proof}

We will see shortly that a tile system is actually an example of a so-called $k$\textit{-rank graph}, (specifically a $2$-rank graph) which were introduced in \cite{KumPas2000} to build on work by \cite{RobSte1999}.

\subsection*{Higher-rank graphs}

\begin{defn}\label{k-graph}
	Let $\Lambda$ be a category such that $\Ob(\Lambda)$ and $\Hom(\Lambda)$ are countable sets (that is, a \textit{countable small category}), and identify $\Ob(\Lambda)$ with the identity morphisms in $\Hom(\Lambda)$. For a morphism $\lambda \in \Hom_\Lambda(u,v)$, we define range and source maps $r(\lambda) = v$ and $s(\lambda) = u$ respectively.
	
	Let $d : \Lambda \rightarrow \mathbb{N}^k$ be a functor, called the \textbf{degree map}, and let $\lambda \in \Hom(\Lambda)$. We call the pair $(\Lambda, d)$ a $k$\textbf{-rank graph} (or simply a $k$\textbf{-graph}) if, whenever $d(\lambda) = \mathbf{m} + \mathbf{n}$ for some $\mathbf{m},\mathbf{n} \in \mathbb{N}^k$, we can find \textit{unique} elements $\mu, \nu \in \Hom(\Lambda)$ such that $\lambda = \nu \mu$, and $d(\mu)=\mathbf{m}$, $d(\nu) = \mathbf{n}$. Note that for $\mu$, $\nu$ to be composable, we must have $r(\mu) = s(\nu)$.
	
	For $\mathbf{n} \in \mathbb{N}^k$, we write $\Lambda^\mathbf{n} := d\inv (\mathbf{n})$; by the above property, we have that $\Lambda^\mathbf{0} = \Ob(\Lambda)$, and we call the elements of $\Lambda^\mathbf{0}$ the \textbf{vertices} of $(\Lambda, d)$. \cite{KumPas2000}
\end{defn}

We direct the reader to e.g. \cite{RaeSimYee2003} for further reference and standard examples of higher-rank graphs, in the event that the reader has not come across them.

%
%
%
%
%

If $E$ is a directed graph on $n$ vertices, we can construct an $n \times n$ vertex matrix $M_E(i,j)$ with $ij$-th entry $1$ if there is an edge from $i$ to $j$, and $0$ otherwise.

If $E$, $F$ are directed graphs with the same vertex set, and such that their associated vertex matrices $M_E$, $M_F$ commute, then \cite{KumPas2000} showed that we can construct a $2$-rank graph out of $E$ and $F$. We use their method to prove:

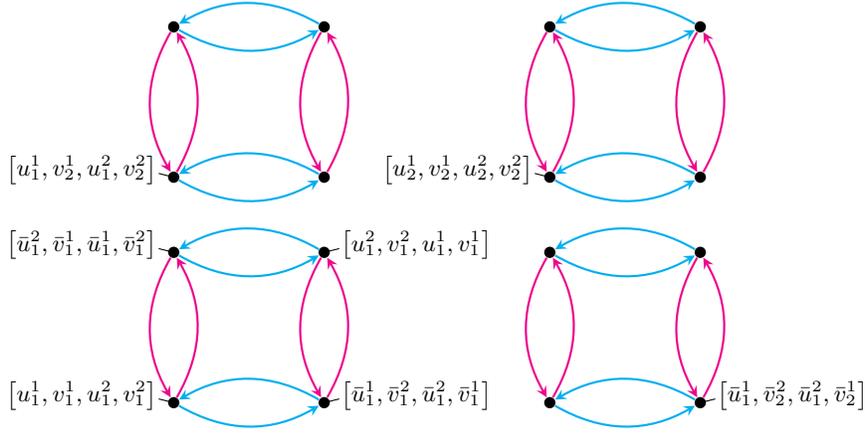
\begin{figure}[h]
	\begin{center}
		\begin{tikzpicture}
		\tikzset{
			big dot/.style={
				circle, inner sep=0pt, 
				minimum size=1.5mm, fill=black
			}
		}
		\node[big dot] (A) at (0,0) {};
		\node[big dot] (B) at (2,0) {};
		\node[big dot] (C) at (0,2) {};
		\node[big dot] (D) at (2,2) {};
		
		\draw[->, thick, >=stealth, cyan, bend right] (A) to (B);
		\draw[->, thick, >=stealth, cyan, bend right] (B) to (A);
		\draw[->, thick, >=stealth, cyan, bend right] (C) to (D);
		\draw[->, thick, >=stealth, cyan, bend right] (D) to (C);
		
		\draw[->, thick, >=stealth, magenta, bend right] (A) to (C);
		\draw[->, thick, >=stealth, magenta, bend right] (C) to (A);
		\draw[->, thick, >=stealth, magenta, bend right] (B) to (D);
		\draw[->, thick, >=stealth, magenta, bend right] (D) to (B);
		
		\draw[anchor=east] (-0.1,0.1) node{\small$\big[ u_1^1,v_1^1,u_1^2,v_1^2\big]$};
		\draw (0,0) -- (-0.2,0.05);
		\draw[anchor=east] (-0.1,2.1) node{\small$\big[ \bar{u}_1^2, \bar{v}_1^1, \bar{u}_1^1, \bar{v}_1^2\big]$};
		\draw (0,2) -- (-0.2,2.05);
		\draw[anchor=west] (2.1,2.1) node{\small$\big[ u_1^2,v_1^2,u_1^1,v_1^1\big]$};
		\draw (2,0) -- (2.2,0.05);
		\draw[anchor=west] (2.1,0.1) node{\small$\big[ \bar{u}_1^1, \bar{v}_1^2, \bar{u}_1^2, \bar{v}_1^1\big]$};
		\draw (2,2) -- (2.2,2.05);
		
		\begin{scope}[xshift=5cm]
		\node[big dot] (A) at (0,0) {};
		\node[big dot] (B) at (2,0) {};
		\node[big dot] (C) at (0,2) {};
		\node[big dot] (D) at (2,2) {};
		
		\draw[->, thick, >=stealth, cyan, bend right] (A) to (B);
		\draw[->, thick, >=stealth, cyan, bend right] (B) to (A);
		\draw[->, thick, >=stealth, cyan, bend right] (C) to (D);
		\draw[->, thick, >=stealth, cyan, bend right] (D) to (C);
		
		\draw[->, thick, >=stealth, magenta, bend right] (A) to (C);
		\draw[->, thick, >=stealth, magenta, bend right] (C) to (A);
		\draw[->, thick, >=stealth, magenta, bend right] (B) to (D);
		\draw[->, thick, >=stealth, magenta, bend right] (D) to (B);
		
		\draw[anchor=west] (2.1,0.1) node{\small$\big[ \bar{u}_1^1, \bar{v}_2^2, \bar{u}_1^2, \bar{v}_2^1\big]$};
		\draw (2,0) -- (2.2,0.05);
		\end{scope}
		
		\begin{scope}[yshift=3cm]
		\node[big dot] (A) at (0,0) {};
		\node[big dot] (B) at (2,0) {};
		\node[big dot] (C) at (0,2) {};
		\node[big dot] (D) at (2,2) {};
		
		\draw[->, thick, >=stealth, cyan, bend right] (A) to (B);
		\draw[->, thick, >=stealth, cyan, bend right] (B) to (A);
		\draw[->, thick, >=stealth, cyan, bend right] (C) to (D);
		\draw[->, thick, >=stealth, cyan, bend right] (D) to (C);
		
		\draw[->, thick, >=stealth, magenta, bend right] (A) to (C);
		\draw[->, thick, >=stealth, magenta, bend right] (C) to (A);
		\draw[->, thick, >=stealth, magenta, bend right] (B) to (D);
		\draw[->, thick, >=stealth, magenta, bend right] (D) to (B);
		
		\draw[anchor=east] (-0.1,0.1) node{\small$\big[ u_1^1,v_2^1,u_1^2,v_2^2\big]$};
		\draw (0,0) -- (-0.2,0.05);
		\end{scope}
		
		\begin{scope}[xshift=5cm,yshift=3cm]
		\node[big dot] (A) at (0,0) {};
		\node[big dot] (B) at (2,0) {};
		\node[big dot] (C) at (0,2) {};
		\node[big dot] (D) at (2,2) {};
		
		\draw[->, thick, >=stealth, cyan, bend right] (A) to (B);
		\draw[->, thick, >=stealth, cyan, bend right] (B) to (A);
		\draw[->, thick, >=stealth, cyan, bend right] (C) to (D);
		\draw[->, thick, >=stealth, cyan, bend right] (D) to (C);
		
		\draw[->, thick, >=stealth, magenta, bend right] (A) to (C);
		\draw[->, thick, >=stealth, magenta, bend right] (C) to (A);
		\draw[->, thick, >=stealth, magenta, bend right] (B) to (D);
		\draw[->, thick, >=stealth, magenta, bend right] (D) to (B);
		
		\draw[anchor=east] (-0.1,0.1) node{\small$\big[ u_2^1,v_2^1,u_2^2,v_2^2\big]$};
		\draw (0,0) -- (-0.2,0.05);
		\end{scope}
		\end{tikzpicture}
	\end{center}
	\caption{Visualisation of the tile system corresponding to the complete bipartite graph $\kappa(2,2)$. Each vertex is labelled with an element of $\mathcal{S}(\kappa)$; a few labels have been shown here. A blue (resp. magenta) arrow joins vertex $A$ to $B$ if and only if $M_1(A,B)=1$ (resp. $M_2(A,B)=1$). Notice the commuting squares, which give the tile system a $2$-rank graph structure: from any vertex $A$, follow a blue arrow, and then a magenta arrow to another vertex $D$, say. Then $\theta$ defines a unique magenta-blue path from $A$ to $D$. The $1$-skeleton of the $2$-rank graph $\Lambda(\kappa(\alpha,\beta))$ is strongly-connected only when $\alpha, \beta \geq 3$.}\label{fig_2-rank}
\end{figure}

\begin{prop}\label{tile_2rank}
	Let $\kappa = \kappa(\alpha,\beta)$ be the complete bipartite graph on $\alpha \geq 2$ white and $\beta \geq 2$ black vertices, and let $(\kappa, U, V, \mathcal{S}(\kappa))$ be a tile system with adjacency matrices $M_1$, $M_2$. This has a $2$-rank graph structure.
\end{prop}

\begin{proof}
	Following the method of Theorem \ref{vdovina}, label the elements of the sets $U$, $V$ such that 
	\begin{align*}
	U &= \big\lbrace u_1^1,u_1^2,\ldots , u_\alpha^1,u_\alpha^2, \bar{u}_1^1,\bar{u}_1^2,\ldots , \bar{u}_\alpha^1,\bar{u}_\alpha^2 \big\rbrace, \\
	\quad V &= \big\lbrace v_1^1,v_1^2,\ldots , v_\beta^1,v_\beta^2, \bar{v}_1^1,\bar{v}_1^2, \ldots , \bar{v}_\beta^1,\bar{v}_\beta^2 \big\rbrace,
	\end{align*}
	where $u_1,\ldots , u_\alpha$ and $v_1,\ldots ,v_\beta$ are the white and black vertices of $\kappa$, respectively. Construct the tile complex $TC(\kappa)$, and consider the set $\mathcal{S}(\kappa) \subseteq U \times V \times U \times V$ of pointed tiles of $TC(\kappa)$. Since $\kappa$ is complete, there is for each $u_i$ and $v_j$ an edge joining them, hence:
\begin{multline*}
\mathcal{S}(\kappa) = \big\lbrace \big[ u_i^1,v_j^1,u_i^2,v_j^2 \big], \big[ \bar{u}_i^1, \bar{v}_j^2, \bar{u}_i^2, \bar{v}_j^1 \big], \\
\big[ u_i^2,v_j^2,u_i^1,v_j^1 \big], \big[ \bar{u}_i^2, \bar{v}_j^1, \bar{u}_i^1, \bar{v}_j^2 \big] \bigm\vert 1 \leq i \leq \alpha, 1\leq j \leq \beta \big\rbrace .
\end{multline*}
Consider the corresponding adjacency matrices $M_1$ and $M_2$ as described in Definition \ref{adjacency_matrices}, and note that they commute, by Proposition \ref{tile_complex_uce}. We can draw directed graphs $E$, $F$ with the same vertex set $E^0 = F^0 = \mathcal{S}(\kappa)$, and a directed edge joining vertex $A$ to $B$ if and only if $M_1(A,B) = 1$, $M_2(A,B) = 1$, respectively (Figure \ref{fig_2-rank}). Write $r_E$, $s_E$ (resp. $r_F$, $s_F$) for the maps describing the respective range and source of edges in $E^1$ (resp. $F^1$). 

Define the following sets: $E^1 \ast F^1 := \big\lbrace (\lambda,\mu) \in E^1 \times F^1 \mid r_E(\lambda) = s_F(\mu) \big\rbrace$ and $F^1 \ast E^1 := \big\lbrace (\mu,\lambda) \in F^1 \times E^1 \mid r_F(\mu) = s_E(\lambda) \big\rbrace$. By the fact that $M_1$, $M_2$ commute, there is a unique bijection $\theta : E^1 \ast F^1 \rightarrow F^1 \ast E^1$ mapping $(\lambda,\mu) \mapsto (\mu',\lambda')$ such that $s_E(\lambda) = s_F(\mu')$ and $r_F(\mu) = r_E(\lambda')$.

We construct a $2$-rank graph $(\Lambda,d)$ as follows: let $\Lambda^0 = \mathcal{S}(\kappa)$, and for $(m,n) \in \mathbb{N}^2$, write $W(m,n) := \big\lbrace (p,q) \in \mathbb{N}^2 \mid p \leq m, q \leq n \big\rbrace$. Then an element of $\Lambda^{(m,n)}$ is given by a triple $(A,\lambda,\mu) = ((A(p,q))_{p,q},(\lambda(p,q))_{p,q},(\mu(p,q))_{p,q})$ such that:
\begin{enumerate}[label=(\alph*)]
	\item  $A(p,q) \in \mathcal{S}(\kappa)$ for some $(p,q) \in W(m,n)$, 
	\item $\lambda(p,q) \in E^1$ for some $(p,q) \in W(m-1,n)$, \item $\mu(p,q) \in F^1$ for some $(p,q) \in W(m,n-1)$, 
	\item $s_E(\lambda(p,q)) = s_F(\mu(p,q)) = A(p,q)$,
	\item $r_E(\lambda(p,q)) = A(p+1,q)$ and $r_F(\mu(p,q)) = A(p,q+1)$,
	\item $\theta(\lambda(p,q),\mu(p+1,q)) = (\mu(p,q),\lambda(p,q+1))$,
\end{enumerate}
whenever these conditions make sense. We write $\Lambda := \bigcup_{m,n \geq 0} \Lambda^{(m,n)}$, and define range and source maps $r(A,\lambda,\mu) := A(0,0)$, $s(A,\lambda,\mu):= A(m,n)$ respectively. We must be wary that two finite paths $\mu$, $\nu$ in such a directed graph $E$ can be concatenated to give a path $\nu \cdot\mu$ if and only if $s_E(\mu) = r_E(\nu)$, so we ``change the direction'' of the source and range of the arrows here.

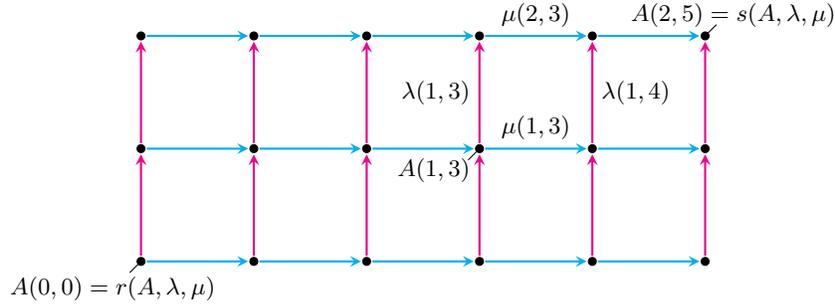
\begin{figure}[h]
	\begin{center}
		\begin{tikzpicture}
		\begin{scope}[scale=0.75]
		\filldraw (0,0) circle (2pt);
		\filldraw[shift={(2,0)}] (0,0) circle (2pt);
		\filldraw[shift={(4,0)}] (0,0) circle (2pt);
		\filldraw[shift={(6,0)}] (0,0) circle (2pt);
		\filldraw[shift={(8,0)}] (0,0) circle (2pt);
		\filldraw[shift={(10,0)}] (0,0) circle (2pt);
		
		\filldraw[shift={(0,2)}] (0,0) circle (2pt);
		\filldraw[shift={(2,2)}] (0,0) circle (2pt);
		\filldraw[shift={(4,2)}] (0,0) circle (2pt);
		\filldraw[shift={(6,2)}] (0,0) circle (2pt);
		\filldraw[shift={(8,2)}] (0,0) circle (2pt);
		\filldraw[shift={(10,2)}] (0,0) circle (2pt);
		
		\filldraw[shift={(0,4)}] (0,0) circle (2pt);
		\filldraw[shift={(2,4)}] (0,0) circle (2pt);
		\filldraw[shift={(4,4)}] (0,0) circle (2pt);
		\filldraw[shift={(6,4)}] (0,0) circle (2pt);
		\filldraw[shift={(8,4)}] (0,0) circle (2pt);
		\filldraw[shift={(10,4)}] (0,0) circle (2pt);
		
		\draw[<-,thick,>=stealth,cyan] (1.9,0) -- (0.1,0);
		\draw[<-,thick,>=stealth,cyan] (3.9,0) -- (2.1,0);
		\draw[<-,thick,>=stealth,cyan] (5.9,0) -- (4.1,0);
		\draw[<-,thick,>=stealth,cyan] (7.9,0) -- (6.1,0);
		\draw[<-,thick,>=stealth,cyan] (9.9,0) -- (8.1,0);
		
		\draw[<-,thick,>=stealth,magenta] (0,1.9) -- (0,0.1);
		\draw[<-,thick,>=stealth,magenta] (0,3.9) -- (0,2.1);
		
		\begin{scope}[xshift=2cm]
		\draw[<-,thick,>=stealth,magenta] (0,1.9) -- (0,0.1);
		\draw[<-,thick,>=stealth,magenta] (0,3.9) -- (0,2.1);
		\end{scope}
		\begin{scope}[xshift=4cm]
		\draw[<-,thick,>=stealth,magenta] (0,1.9) -- (0,0.1);
		\draw[<-,thick,>=stealth,magenta] (0,3.9) -- (0,2.1);
		\end{scope}
		\begin{scope}[xshift=6cm]
		\draw[<-,thick,>=stealth,magenta] (0,1.9) -- (0,0.1);
		\draw[<-,thick,>=stealth,magenta] (0,3.9) -- (0,2.1);
		\end{scope}
		\begin{scope}[xshift=8cm]
		\draw[<-,thick,>=stealth,magenta] (0,1.9) -- (0,0.1);
		\draw[<-,thick,>=stealth,magenta] (0,3.9) -- (0,2.1);
		\end{scope}
		\begin{scope}[xshift=10cm]
		\draw[<-,thick,>=stealth,magenta] (0,1.9) -- (0,0.1);
		\draw[<-,thick,>=stealth,magenta] (0,3.9) -- (0,2.1);
		\end{scope}
		
		\begin{scope}[yshift=2cm]
		\draw[<-,thick,>=stealth,cyan] (1.9,0) -- (0.1,0);
		\draw[<-,thick,>=stealth,cyan] (3.9,0) -- (2.1,0);
		\draw[<-,thick,>=stealth,cyan] (5.9,0) -- (4.1,0);
		\draw[<-,thick,>=stealth,cyan] (7.9,0) -- (6.1,0);
		\draw[<-,thick,>=stealth,cyan] (9.9,0) -- (8.1,0);
		\end{scope}
		\begin{scope}[yshift=4cm]
		\draw[<-,thick,>=stealth,cyan] (1.9,0) -- (0.1,0);
		\draw[<-,thick,>=stealth,cyan] (3.9,0) -- (2.1,0);
		\draw[<-,thick,>=stealth,cyan] (5.9,0) -- (4.1,0);
		\draw[<-,thick,>=stealth,cyan] (7.9,0) -- (6.1,0);
		\draw[<-,thick,>=stealth,cyan] (9.9,0) -- (8.1,0);
		\end{scope}		
		
		\draw[anchor=south] (7,2) node{\small$\mu(1,3)$};
		\draw[anchor=east] (6,3) node{\small$\lambda(1,3)$};
		\draw[anchor=south] (7,4) node{\small$\mu(2,3)$};
		\draw[anchor=west] (8,3) node{\small$\lambda(1,4)$};
		
		\draw[anchor=north east] (6,2) node{\small$A(1,3)$};
		\draw (6,2) -- (5.8,1.8);
		
		\draw[anchor=north] (-0.5,-0.1) node{\small$A(0,0) = r(A,\lambda,\mu)$};
		\draw (0,0) -- (-0.2,-0.2);
		
		\draw[anchor=south] (10.5,4) node{\small$A(2,5) = s(A,\lambda,\mu)$};
		\draw (10,4) -- (10.2,4.2);
		
		\end{scope}
		\end{tikzpicture}
	\end{center}
	\caption{An element $(A,\lambda,\mu)$ of $\Lambda^{(m,n)}$ can be represented as an $m \times n$ grid. The isomorphism $\theta$ defines commuting squares. Here is an element of $\Lambda^{(2,5)}$.}\label{fig_2-rank_squares}
\end{figure} 

If $\varphi$, $\psi$ are paths of nonzero length $m$, $n$ in $E$, $F$ respectively, with $r_E(\varphi) = s_F(\psi)$, then there is a unique element $\varphi \psi = (A,\lambda,\mu) \in \Lambda^{(m,n)}$ such that $\varphi = \lambda(0,0) \cdots \lambda(m-1,0)$, and $\psi = \mu(m,0) \cdots \mu(m,n-1)$. If instead (or as well) $r_F(\psi) = s_E(\varphi)$, then there is a unique element $\psi \varphi$ such that $\varphi = \lambda(0,n) \cdots \lambda(m-1,n)$ and $\psi = \mu(0,0) \cdots \mu(0,n-1)$ (Figure \ref{fig_2-rank_squares}).

Then, given two elements $( A_1,\lambda_1,\mu_1) \in \Lambda^{( m_1,n_1)}$ and $( A_2,\lambda_2, \mu_2) \in \Lambda^{( m_2,n_2)}$ such that $A_1 ( m_1,n_1) = A_2(0,0)$, we can find a unique element $( A_1,\lambda_1,\mu_1)( A_2,\lambda_2,\mu_2) = ( A_3,\lambda_3,\mu_3 )$ in $\Lambda^{( m_1+m_2,n_1+n_2)}$ such that:
\begin{enumerate}[label=(\alph*)]
	\item $A_3(p,q) = A_1(p,q)$, and $A_3(m+p,n+q) = A_2(p,q)$,
	\item $\lambda_3(p,q) = \lambda_1(p,q)$, and $\lambda_3(m+p,n+q) = \lambda_2(p,q)$,
	\item $\mu_3(p,q) = \mu_1(p,q)$, and $\mu_3(m+p,n+q) = \mu_2(p,q)$,
\end{enumerate}
whenever these conditions make sense. In this way, composition is defined in $\Lambda$, and by construction we have associativity and the factorisation property of Definition \ref{k-graph}. Thus $\Lambda$, together with obvious degree functor $d : (A,\lambda,\mu) \mapsto (m,n)$ for $(A, \lambda, \mu) \in \Lambda^{(m,n)}$, has the structure of a $2$-rank graph, and we write $(\Lambda,d) = \Lambda(\kappa)$.
\end{proof}

\begin{defn}
	Let $(\Lambda,d)$ be a $k$-rank graph, let $\mathbf{n} \in \mathbb{N}^k$, and let $v \in \Lambda^\mathbf{0}$. Write $\Lambda^\mathbf{n}(v)$ for the set of morphisms in $\Lambda^\mathbf{n}$ which map onto the vertex $v$, that is, $\Lambda^\mathbf{n}(v) := \lbrace \lambda \in \Lambda^\mathbf{n} \mid r(\lambda) = v\rbrace$. We say that $(\Lambda,d)$ is \textbf{row-finite} if each set $\Lambda^\mathbf{n}(v)$ is finite, and that $(\Lambda,d)$ has \textbf{no sources} if each $\Lambda^\mathbf{n}(v)$ is non-empty.
\end{defn}

As an extension of the concept of a \textit{graph algebra} (c.f. \cite{Rae2005}), we can associate a $\cst$-algebra to a $k$-rank graph as follows:

\begin{defn}\label{graph_algebra}
	Let $\Lambda = (\Lambda,d)$ be a row-finite $k$-rank graph with no sources. We define $\cst(\Lambda)$ to be the universal $\cst$-algebra generated by a family $\lbrace s_\lambda \mid \lambda \in \Lambda\rbrace$ of \textit{partial isometries} which have the following properties:
	\begin{enumerate}[label=(\alph*)]
		\item The set $\big\lbrace s_v \mid v \in \Lambda^\mathbf{0}\big\rbrace$ satisfies $(s_v)^2 = s_v = s_v^*$ and $s_u s_v = 0$ for all $u \neq v$.
		\item If $r(\lambda) = s(\mu)$ for some $\lambda, \mu \in \Lambda$, then $s_{\mu\lambda} = s_\mu s_\lambda$.
		\item For all $\lambda \in \Lambda$, we have $s_\lambda^* s_\lambda = s_{s(\lambda)}$.
		\item For all vertices $v \in \Lambda^\mathbf{0}$ and $\mathbf{n} \in \mathbb{N}^k$, we have:
		\[
		s_v = \sum_{\lambda \in \Lambda^\mathbf{n}(v)} s_\lambda s_\lambda^*.
		\]
	\end{enumerate}
	Note that without the row-finiteness condition, property (d) is not well-defined.
\end{defn}

\begin{thm}[Evans, 2008]\label{evans}
	Let $\Lambda$ be a row-finite $2$-graph with no sources, finite vertex set $\Lambda^\mathbf{0}$ with $\big| \Lambda^\mathbf{0}\big| = n$, and vertex matrices $M_E$, $M_F$. Then:
	\begin{align*}
	K_0 (\cst (\Lambda)) &\cong \mathbb{Z}^{r_0} \oplus \tors\big(\coker\big(\mathbf{1}-M_E^T, \mathbf{1}-M_F^T\big)\big), \\
	K_1(\cst (\Lambda)) &\cong \mathbb{Z}^{r_1} \oplus \tors\big(\coker\big(\mathbf{1}-M_E, \mathbf{1}-M_F\big)\big),
	\end{align*}
	where
	\begin{align*}
	r_0 &:= \rk\big(\coker\big(\mathbf{1}-M_E^T, \mathbf{1}-M_F^T\big)\big) + \rk\big(\coker \big(\mathbf{1}-M_E,\mathbf{1}-M_F\big) \big), \\
	r_1 &:= \rk\big(\coker\big(\mathbf{1}-M_E^T, \mathbf{1}-M_F^T\big)\big) + \rk\big(\coker\big(\mathbf{1}-M_E,\mathbf{1}-M_F\big) \big),
	\end{align*}
	and where $\mathbf{1}$ is the $n \times n$ identity matrix, $({}\ast {} ,{} \ast {})$ denotes the corresponding block $n \times 2n$ matrix, $\rk(\mathfrak{G})$ denotes the torsion-free rank of a finitely-generated Abelian group $\mathfrak{G}$, and $\tors(\mathfrak{G})$ denotes the torsion part of $\mathfrak{G}$. \cite[Proposition 4.4]{Eva2008}
\end{thm}

\begin{corl}\label{evans_corl}
	Let $\kappa = \kappa(\alpha,\beta)$ be the complete bipartite graph on $\alpha \geq 2$ white and $\beta \geq 2$ black vertices, and let $(\kappa,U,V,\mathcal{S}(\kappa))$ be a tile system with adjacency matrices $M_1$, $M_2$ as in Definition \ref{adjacency_matrices}. As an abuse of notation, we write $\cst(\kappa) = \cst(\Lambda(\kappa))$. Then
	\[
	K_0 (\cst (\kappa)) = K_1 (\cst (\kappa)) = \coker\big(\mathbf{1}-M_1^T, \mathbf{1}-M_2^T\big) \oplus \rk \big(\coker\big(\mathbf{1}-M_1^T, \mathbf{1}-M_2^T\big)\big).
	\]
\end{corl}

\begin{proof}
	Firstly, $\alpha, \beta < \infty$ by assumption, and by the UCE Property of the tile system (Proposition \ref{tile_complex_uce}) we know that each row and column of $M_1$ and $M_2$ has at least one nonzero element. Hence $\Lambda(\kappa)$ is row-finite, has no sources, and is such that $\big|\Lambda(\kappa)^\mathbf{0}\big| = 4\alpha\beta$, whence the result follows from Theorem \ref{evans}.
\end{proof}

\begin{thm}[K-groups for pointed tile systems]\label{pointed}
	Let $a,b \geq 0$, and let $\kappa (a+2,b+2)$ be the complete bipartite graph on $a+2$ white and $b+2$ black vertices. Without loss of generality, we assume that $a\leq b$. Write $l := \lcm ( a, b )$, and $g := \gcd ( a, b )$. Then, for $\epsilon = 0,1$:
	\begin{enumerate}[label=(\roman*)]
		\item If $a=b=0$, then $K_\epsilon (\cst (\kappa (a+2,b+2)) = K_\epsilon (\cst (\kappa (2,2))) \cong \mathbb{Z}^8$.
		
		\item If $a=0,1$ and $b \geq 1$, then
		\[
		K_\epsilon (\cst (\kappa (a+2,b+2))) \cong (\mathbb{Z}/b)^2 \oplus \mathbb{Z}^{4(b+1)}.
		\]
		
		\item If $a,b \geq 2$ and $a,b$ are coprime, then
		\[
		K_\epsilon (\cst (\kappa (a+2,b+2))) \cong (\mathbb{Z}/a)^{b-a} \oplus ( \mathbb{Z}/ab )^{a+1} \oplus \mathbb{Z}^{2(a+1)(b+1)}.
		\]
		
		\item If $a,b \geq 2$ and $a,b$ are \emph{not} coprime, then
		\[
		K_\epsilon(\cst (\kappa(a+2,b+2))) \cong (\mathbb{Z}/a)^{b-a} \oplus ( \mathbb{Z}/l )^{a+1} \oplus ( \mathbb{Z}/g)^{a+2} \oplus \mathbb{Z}^{2(a+1)(b+1)},
		\]
		where $(\mathbb{Z}/a)^0$ is defined to be the trivial group in the case that $a=b$.
	\end{enumerate}
\end{thm}

\begin{proof}
	We begin by proving (iii) and (iv), since (i) and (ii) are special cases thereof.
	
	So, assume that $a,b \geq 2$. Write $\alpha = a+2$ and $\beta = b+2$, and for $1 \leq i \leq \alpha$, $1 \leq j \leq \beta$, let $A_{ij}$ denote the pointed tile $\big[ u_i^1,v_j^1,u_i^2,v_j^2\big] \in \mathcal{S}(\kappa)$. Similarly, write $B_{ij} := \big[ \bar{u}_i^1, \bar{v}_j^2, \bar{u}_i^2, \bar{v}_j^1 \big]$, $C_{ij} := \big[ \bar{u}_i^2, \bar{v}_j^1, \bar{u}_i^1, \bar{v}_j^2 \big]$, $D_{ij} := \big[ u_i^2,v_j^2,u_i^1,v_j^1 \big]$ for the tiles with the same edge labels as the horizontal reflection, vertical reflection, and rotation by $\pi$ of $A_{ij}$, respectively. Then $\mathcal{S}(\kappa) = \lbrace A_{ij},B_{ij},C_{ij},D_{ij} \mid 1 \leq i \leq \alpha, 1 \leq j \leq \beta \rbrace$, and
	\begin{multline}\label{group1}
	\coker = \coker\big(\mathbf{1}-M_1^T,\mathbf{1}-M_2^T\big) = \Bigg\langle S \in \mathcal{S}(\kappa) \Biggm\vert S = \sum_{T \in \mathcal{S}(\kappa)} M_1(S,T)\cdot T \\
	= \sum_{T \in \mathcal{S}(\kappa)} M_2(S,T)\cdot T \Bigg\rangle .
	\end{multline}
	Now fix $p \in \lbrace 1 ,\ldots , \alpha\rbrace$, $q \in \lbrace 1 , \ldots , \beta\rbrace$, and notice that:
	\begin{itemize}
		\item $M_1(A_{pq}, T) = 1$ iff $T = B_{iq}$; $M_1(B_{pq}, T) = 1$ iff $T = A_{iq}$, for some $i \neq p$,
		\item $M_1(C_{pq}, T) = 1$ iff $T = D_{iq}$; $M_1(D_{pq}, T) = 1$ iff $T = C_{iq}$, for some $i \neq p$,
		
		\item $M_2(A_{pq}, T) = 1$ iff $T = C_{pj}$; $M_2(B_{pq}, T) = 1$ iff $T = D_{pj}$, for some $j \neq q$,
		\item $M_2(C_{pq}, T) = 1$ iff $T = A_{pj}$; $M_2(D_{pq}, T) = 1$ iff $T = B_{pj}$, for some $j \neq q$.
	\end{itemize}
	Hence the relations of (\ref{group1}) are given by equations of the form $A_{pq} = \sum_{i \neq p} B_{iq} = \sum_{j \neq q} C_{pj}$, and so on for each $B_{pq}$, $C_{pq}$, and $D_{pq}$. 
	
	In particular, we can write $B_{pq} = \sum_{i \neq p} A_{iq}$ and $C_{pq} = \sum_{j \neq q} A_{pj}$ so that
	\[
	A_{pq} = (\alpha - 1)A_{pq} + (\alpha - 2)\sum_{i \neq p}A_{iq} \quad\text{and}\quad A_{pq} = (\beta - 1)A_{pq} + (\beta - 2)\sum_{j \neq q}A_{pj}.
	\]
	Define $J_q := \sum_{i=1}^\alpha A_{iq}$, and $I_p := \sum_{j=1}^\beta A_{pj}$. Then $(\alpha - 2)J_q = (\beta - 2)I_p = 0$, and viewing the sum of all the tiles $A_{ij}$ both as the sum of all the $I_i$ and of the $J_j$, we conclude also that $g\Sigma = 0$, where $\Sigma := \sum_{i,j} A_{ij}$.
	
	Now, we can also write $D_{pq}$ (and all of the relevant relations) in terms of the $A_{ij}$, namely $D_{pq} = \sum_{i \neq p} \sum_{j \neq q} A_{ij}$. Hence we can remove all the $B_{pq}$, $C_{pq}$, and $D_{pq}$ from the list of generators of $\coker$, such that
	\begin{multline}\label{eq_presentation1}
	\coker = \big\langle A_{pq} \bigm\vert (\alpha -2)J_q = (\beta -2)I_p = 0, J_q = \textstyle\sum_i A_{iq}, I_p = \textstyle\sum_j A_{pj}, \\
	\text{ for } 1 \leq p \leq \alpha, 1 \leq q \leq \beta \big\rangle .
	\end{multline}
	We have the following equalities: 
	\[
	A_{p1} = I_p - \sum_{j=2}^\beta A_{pj},\quad
	A_{1q} = J_q - \sum_{i=2}^\alpha A_{iq},\quad
	I_1 = \Sigma - \sum_{i=2}^\alpha I_i,\quad
	J_1 = \Sigma - \sum_{j=2}^\beta J_j.
	\]
	Furthermore, $A_{11}$ may be expressed in terms of $\Sigma$, $I_p$, $J_q$, and $A_{pq}$ for $p,q \geq 2$, and so after a sequence of Tietze transformations on (\ref{eq_presentation1}), we find that
	\begin{equation}\label{eq_presentation2}
	\coker = \langle \Sigma, I_p, J_q, A_{pq} \mid (\alpha -2)J_q = (\beta -2)I_p = g\Sigma = 0, \text{ for } 2 \leq p \leq \alpha, 2 \leq q \leq \beta\rangle ,
	\end{equation}
	where $g:= \gcd(\alpha -2,\beta -2)$. This, after substituting $a = \alpha -2$, $b = \beta -2$, gives a presentation for $(\mathbb{Z}/b )^{a+1} \oplus ( \mathbb{Z}/a )^{b+1} \oplus ( \mathbb{Z}/g) \oplus \mathbb{Z}^{(a+1)(b+1)}$. In particular, we have $a+1$ copies of $( \mathbb{Z}/b ) \oplus ( \mathbb{Z}/a )$.  It is well-known that if $a$ and $b$ are not coprime, $( \mathbb{Z}/b ) \oplus ( \mathbb{Z}/a ) \cong ( \mathbb{Z}/l ) \oplus ( \mathbb{Z}/g )$; in case (iv), this together with Corollary \ref{evans_corl} immediately gives the desired result. In case (iii), where $a$ and $b$ are coprime, we instead have that $( \mathbb{Z}/b ) \oplus ( \mathbb{Z}/a ) \cong ( \mathbb{Z}/ab)$, and we are done.
	
	Now consider case (i), where $\alpha = \beta = 2$. Then, following the method above, $\coker$ is generated by $\lbrace A_{pq} \mid p,q = 1,2 \rbrace$ with trivial relations, and so $\coker \cong \mathbb{Z}^4$. Hence by Corollary \ref{evans_corl}, $K_\epsilon(\cst(\kappa)) \cong \mathbb{Z}^8$.
	
	Similarly, when $\alpha = 2$ and $\beta \geq 3$, it is straightforward to show that
	\[
	\coker = \langle I_p, A_{pq} \mid (\beta - 2)I_p = 0, \text{ for } p =1,2 \text{ and } 2 \leq q \leq \beta\rangle,
	\]
	and when $\alpha = 3$ and $\beta \geq 3$, we have
	\[
	\coker = \langle \Sigma,I_p,J_q, A_{pq} \mid J_q = (\beta - 2)I_p = \Sigma = 0, \text{ for } p =2,3 \text{ and } 2 \leq q \leq \beta\rangle,
	\]
	both of which are presentations for $(\mathbb{Z}/(\beta - 2))^2 \oplus \mathbb{Z}^{2(\beta - 1)}$; hence by Corollary \ref{evans_corl}, (ii) is proved.
\end{proof}

\begin{figure}[h]
	\begin{center}
		\begin{tikzpicture}
		\tikzset{
			big dot/.style={
				circle, inner sep=0pt, 
				minimum size=1.5mm, fill=black
			}
		}
		\node[big dot] (A) at (0,0) {};
		\node[big dot] (B) at (2,0) {};
		
		\node[big dot] (C) at (4,0) {};
		\node[big dot] (D) at (6,0) {};
		
		\draw[->, thick, >=stealth, cyan, bend right] (A) to (B);
		\draw[->, thick, >=stealth, cyan, bend right] (B) to (A);
		\draw[->, thick, >=stealth, cyan, bend right] (C) to (D);
		\draw[->, thick, >=stealth, cyan, bend right] (D) to (C);
		
		\draw[] (3,0) node{$\times$};
		\draw[] (1,0.6) node{$C_2$};
		\draw[] (5,0.6) node{$C_2$};
		\end{tikzpicture}
		\caption{The $2$-graph $\Lambda(\kappa(2,2))$, depicted in Figure \ref{fig_2-rank}, consists of four copies of $C_2 \times C_2$, where $C_2$ is the cyclic $1$-graph with two vertices.}\label{fig_example}
	\end{center}
\end{figure}
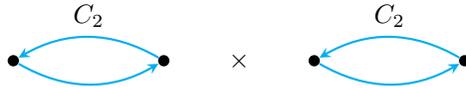

\begin{ex}
	Recall the tile system corresponding to $\kappa(2,2)$, given as an example in Figure \ref{fig_2-rank}. From the diagram, we can see that the ($1$-skeleton of the) $2$-rank graph $\Lambda(\kappa(2,2))$ comprises four connected components, each being the Cartesian product $C_2 \times C_2$, depicted in Figure \ref{fig_example}. It is well known that the $k$-graph $\cst$-algebra of $C_2$ is isomorphic to $M_2(C(\mathbb{T}))$. Furthermore, there is a natural isomorphism $\cst(C_m \times C_n) \cong M_{mn}(C(\mathbb{T}^2))$, and so $\cst(\kappa(2,2)) \cong (M_4(C(\mathbb{T}^2)))^4$. The K-groups of this $\cst$-algebra are both $\mathbb{Z}^8$, in agreement with Theorem \ref{pointed}.
\end{ex}

\begin{thm}\label{pointed_identity}
	Let $\alpha, \beta \geq 3$, and let $\kappa = \kappa (\alpha, \beta)$ be the complete bipartite graph on $\alpha$ white and $\beta$ black vertices. Then the order of the class of the identity $[\mathbf{1}]$ in $K_0(\cst(\Lambda(\kappa)))$ is equal to $g := \gcd(\alpha - 2, \beta - 2)$.
\end{thm}

\begin{proof}
	From \cite{KimRob2002}, it follows that the order of $[\mathbf{1}]$ in $K_0(\cst(\kappa))$ is equal to the order of the sum of pointed tiles in $\mathcal{S}(\kappa)$; by considerations in the proof of Theorem \ref{pointed}, we know this to be $g$.
\end{proof}

\section{Aperiodicity and Kirchberg-Phillips Classification}\label{S_aperiodicity}

Kumjian and Pask in \cite{KumPas2000} have developed conditions under which the $\cst$-algebra of a $k$-rank graph is both simple and purely-infinite. In this section we show that the conditions are satisfied by the algebras $\cst(\kappa)$, and thus, by Kirchberg and Phillips \cite{KirPrep}, \cite{Phi2000}, that the $\cst(\kappa)$ are completely classified by their K-theory. We detail the following definitions from \cite{KumPas2000}.

Let $k \geq 1$, and let $\Omega_k$ be the countable small category defined by object set $\Ob(\Omega_k) := \mathbb{N}_0^k$ and morphism set
\[
\Hom(\Omega_k) := \big\lbrace (\mathbf{m}, \mathbf{n}) = (m_1, \ldots , m_k, n_1,\ldots , n_k) \in \mathbb{N}_0^k \times \mathbb{N}_0^k \bigm\vert  m_i \leq n_i \text{ for all } 1 \leq i \leq k \big\rbrace .
\]
We identify $\Ob(\Omega_k)$ with the set of identity morphisms $\big\lbrace (\mathbf{m}, \mathbf{m}) \bigm\vert \mathbf{m} \in \mathbb{N}_0^k \big\rbrace$, and hence identify $\Omega_k$ with $\Hom(\Omega_k)$. Define range and source maps $r(\mathbf{m},\mathbf{n}) := \mathbf{m}$ and $s(\mathbf{m},\mathbf{n}) := \mathbf{n}$, respectively. Then $\Omega_k$ together with the degree map $d(\mathbf{m},\mathbf{n}) := \mathbf{n} - \mathbf{m}$ is a $k$-rank graph, which we can visualise as a non-negative integer lattice in $\mathbb{R}^k$ (c.f. Figure 6).

\begin{defn}\label{period}
	Let $\Lambda$ be a $k$-rank graph. We define the \textbf{infinite path space} $\Lambda^\infty$ of $\Lambda$ to be $\Lambda^\infty := \lbrace \varphi: \Omega_k \rightarrow \Lambda \mid \varphi \text{ is a $k$-graph morphism} \rbrace$. 
	
	Given a vertex $v \in \Lambda^{\mathbf{0}}$, we write $\Lambda^\infty(v)$ for the set of infinite paths which begin at $v$, that is, $\Lambda^\infty(v) := \lbrace \varphi \in \Lambda^\infty \mid \varphi(\mathbf{0}) = v \rbrace$.
	
	Let $\mathbf{p} \in \mathbb{Z}^k$, and let $\varphi \in \Lambda^\infty$. We say that $\mathbf{p}$ is a \textbf{period} for $\varphi$ if, for every $(\mathbf{m},\mathbf{n}) \in \Omega_k$ with $\mathbf{m} + \mathbf{p} \geq \mathbf{0}$, we have $\varphi(\mathbf{m}+\mathbf{p}, \mathbf{n}+\mathbf{p}) = \varphi(\mathbf{m}, \mathbf{n})$. We call $\varphi$ \textbf{periodic} if we can find a nonzero period.
	
	Given some $\mathbf{q} \in \mathbb{N}_0^k$ and a path $\varphi \in \Lambda^\infty$, we write $\varphi_\mathbf{q}(\mathbf{m},\mathbf{n}) := (\mathbf{m}+\mathbf{q}, \mathbf{n}+\mathbf{q})$. We say that $\varphi$ is \textbf{eventually periodic} if we can find some nonzero $\mathbf{q} \in \mathbb{N}_0^k$ such that $\varphi_\mathbf{q}$ is periodic. We say that an infinite path $\varphi$ is \textbf{aperiodic} if it is neither periodic nor eventually periodic.
	
	We say that $\Lambda$ satisfies the \textbf{Aperiodicity Condition} (also referred to in the literature as \textbf{Condition (A)}) if, for every vertex $v \in \Lambda^\mathbf{0}$, we can find an aperiodic path $\varphi \in \Lambda^\infty(v)$.
	
	We say that $\Lambda$ is \textbf{cofinal} if, for every vertex $v \in \Lambda^\mathbf{0}$ and every infinite path $\varphi \in \Lambda^\infty$, we can find $\lambda \in \Lambda$ and $\mathbf{n} \in \mathbb{N}_0^k$ such that $r(\lambda) = v$ and $s(\lambda) = \varphi(\mathbf{n})$.
\end{defn}

The Aperiodicity Condition is a generalisation of the condition on $1$-graphs that every cycle have an exit. Similarly, cofinality is a generalisation of the property that every vertex in a $1$-graph can be reached from somewhere on every infinite path.

\begin{lem}
	Consider the complete bipartite graph $\kappa = \kappa(\alpha,\beta)$ for $\alpha,\beta \geq 3$, and let $\Lambda(\kappa)$ be the corresponding $2$-rank graph as constructed by Proposition \ref{tile_2rank}. Then $\Lambda(\kappa)$ satisfies the Aperiodicity Condition.
\end{lem}

\begin{figure}
	\begin{center}
		\includegraphics[scale=0.7]{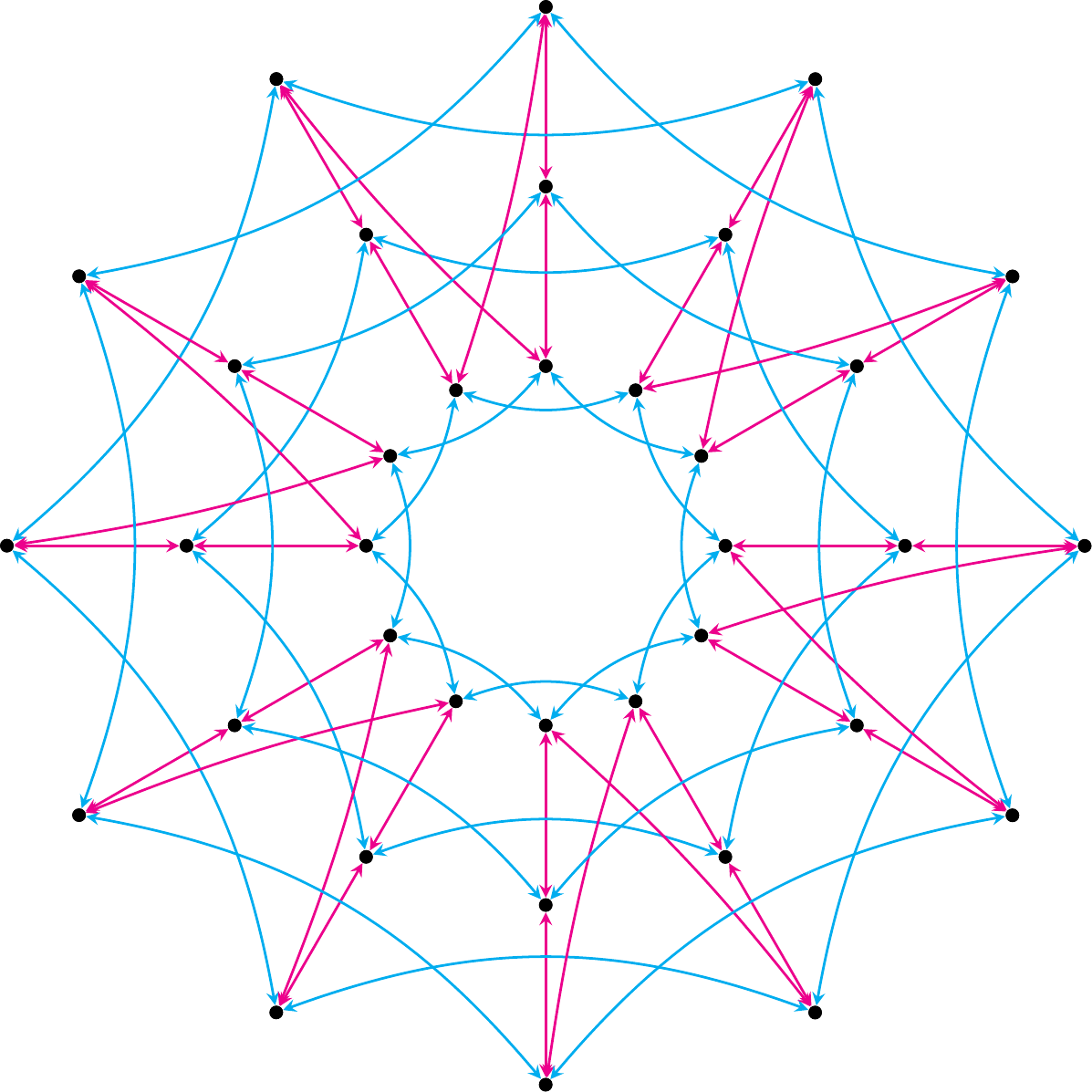}
	\end{center}
\caption{A representation of $\Lambda(\kappa(3,3))$. It is always possible to exit a cycle.}\label{fig_3,3}
\end{figure}

In order to get a feeling as to why this is true, consider Figure \ref{fig_3,3}, which shows a representation of $\Lambda(\kappa(3,3))$. Each vertex is labelled by a pointed tile from $\mathcal{S}(\kappa(3,3))$, and since each tile is vertically-adjacent to two others (and horizontally-adjacent to two others), there are two blue arrows and two magenta arrows emanating from each vertex of $\Lambda(\kappa(3,3))$. This suggests that, analogously to the $1$-graph condition, we can always find an exit to some cycle in $\Lambda$, namely by stopping mid-cycle at a vertex, and diverting the path down the second of the two available edges. Hence, as long as $\alpha, \beta \geq 3$, there will be enough choice at each vertex to be able to exit a cycle.

\begin{proof}
	Firstly, write $\Lambda = \Lambda(\kappa)$, let $A \in \Lambda^\mathbf{0}$ be an arbitrary vertex. We construct an aperiodic infinite path beginning from $A$ in the following way:
	
	Let $x:\Omega_1 \rightarrow \bigcup_{m \geq 0} \Lambda^{(m,0)}$ be a $1$-graph morphism such that $x(0) = A$. The vertex $A$ represents a pointed tile in $\mathcal{S}(\kappa)$, which is horizontally-adjacent to $\beta - 1$ other pointed tiles. Hence $A$ is connected by bidirectional blue arrows to $\beta - 1$ other vertices in $\Lambda$. Choose two of these vertices, $B_1$ and $B_2$, say, and let $x$ be such that
	\[
	x(m,m) = \begin{cases}
	A	&	\text{if $m$ is even,} \\
	B_1	&	\text{if $m = r^2 + r + 1$, for some $r \geq 1$,} \\
	B_2	&	\text{otherwise,}
	\end{cases}
	\]
	for all $m \in \mathbb{N}_0$. Since this forms an aperiodic sequence, there is no $p \in \mathbb{Z}$ such that $x(m,m) = x(m+p,m+p)$ for all $m$, nor any $q \in \mathbb{N}$ such that $x_q$ is periodic; hence $x$ is an aperiodic path. Similarly, define $y : \Omega_1 \rightarrow \bigcup_{n \geq 0} \Lambda^{(0,n)}$ by
	\[
	y(n,n) = \begin{cases}
	A	&	\text{if $n$ is even,} \\
	C_1	&	\text{if $n = s^2 + s + 1$, for some $s \geq 1$,} \\
	C_2	&	\text{otherwise,}
	\end{cases}
	\]
	for some vertices labelled by pointed tiles $C_1$, $C_2$ which are vertically-adjacent to $A$. Then $y$ is also an aperiodic path.	By the UCE Property, $x$ and $y$ uniquely determine an infinite path $\varphi : \Omega_2 \rightarrow \Lambda$ with $\varphi((m,0),(m,0)) = x(m,m)$ and $\varphi((0,n),(0,n)) = y(n,n)$.
	
	Let $D$ denote the unique pointed tile (other than $A$) adjacent to both $B_1$ and $C_1$. This cannot also be adjacent to $B_2$, nor to $C_2$, and so $\varphi((m,n),(m,n)) = D$ precisely when $m = r^2+r+1$ and $n=s^2+s+1$, for some $r,s \geq 1$. As above, there is no $\mathbf{p} \in \mathbb{Z}^2$ such that $\varphi((m,n),(m,n)) = \varphi((m,n)+\mathbf{p},(m,n)+\mathbf{p})$, nor any $\mathbf{q} \in \mathbb{N}_0^2$ such that $\varphi_\mathbf{q}$ is periodic. Since our initial vertex $A$ was arbitrary, we are done.
\end{proof}

The following definitions will be required for the rest of the section. For the reader who desires more detail, we recommend \cite[Chapter 5]{RorLarLau2000}.

\begin{defn}
	Let $\mathcal{A}$ be a unital $\cst$-algebra, and let $\mathcal{B} \subset \mathcal{A}$ be a $\cst$-subalgebra. We say that $\mathcal{B}$ is \textbf{hereditary} if, for all $a,b \in \mathcal{A}$, if $b \in \mathcal{B}$ and $a \leq b$, then $a \in \mathcal{B}$.	
	
	We say that $\mathcal{A}$ is \textbf{simple} if it has no non-trivial closed two-sided ideals.
	
	If $\mathcal{A}$ is simple, we say that it is \textbf{purely-infinite} if every nonzero hereditary $\cst$-subalgebra of $\mathcal{A}$ contains a projection which is Murray-von Neumann equivalent to a proper subprojection of itself. Equivalently, $\mathcal{A}$ is purely-infinite if every nonzero hereditary $\cst$-subalgebra contains a projection equivalent to $1$.
\end{defn}

\begin{thm}[Kumjian-Pask]\label{kumjian-pask_criteria}
	Let $\Lambda$ be a $k$-rank graph which satisfies the Aperiodicity Condition. Then the associated universal $\cst$-algebra $\cst(\Lambda)$ is simple if and only if $\Lambda$ is cofinal.
\end{thm}

\begin{thm}[Kumjian-Pask, Sims]\label{kps}
	Let $\Lambda$ be a $k$-rank graph which is cofinal and which satisfies the Aperiodicity Condition. Suppose that, for every $v \in \Lambda^\mathbf{0}$, we can find $\lambda \in \Lambda$ with $r(\lambda) = v$, and some cycle $\mu \in \Lambda$ with an entrance, such that $d(\mu) \neq \mathbf{0}$, and $s(\lambda) = r(\mu) = s(\mu)$. Then $\cst(\Lambda)$ is purely-infinite. \cite[Proposition 8.8]{Sim2006}
\end{thm}

\begin{prop}\label{simple_purely-inf}
	Consider $\kappa = \kappa(\alpha,\beta)$ for $\alpha,\beta \geq 3$, and let $\Lambda(\kappa)$ be the corresponding $2$-rank graph. Then the corresponding $\cst$-algebra $\cst(\kappa)$ from Definition \ref{graph_algebra} is simple and purely-infinite. 
\end{prop}

\begin{proof}
	Firstly, we observe that $\Lambda(\kappa)$ is cofinal, since the $1$-skeleton of $\Lambda(\kappa)$ is strongly-connected. Hence from Theorem \ref{kumjian-pask_criteria} it follows that $\cst(\kappa)$ is simple. 
	
	Now, let $A \in \Lambda(\kappa)^\mathbf{0}$ be an arbitrary vertex. Since each edge of the $1$-skeleton of $\Lambda(\kappa)$ is bidirectional, we can set $\mu$ to be a path which begins at $A$ and traverses a single blue edge to some vertex $B$, before immediately returning to $A$. Then $d(\mu) = (2,0)$, and since $\alpha, \beta \geq 3$, $B$ is the range of some other blue edge, and so $\mu$ is a cycle with an entrance. Then by strong-connectedness, the conditions of Theorem \ref{kps} are satisfied, and so $\Lambda(\kappa)$ is purely-infinite.
\end{proof}

We make use of the following theorem from \cite{KirPrep}, \cite{Phi2000}. For detail of the definitions, consult e.g. \cite{EffRua2000}, \cite{RorLarLau2000}.

\begin{thm}[Kirchberg-Phillips Classification]\label{kirchberg-phillips}
	Let $\mathcal{A}$ be a separable, nuclear, unital, purely-infinite, simple $\cst$-algebra, which satisfies the Rosenberg-Schochet Universal Coefficient Theorem \cite{RosSch1987}. Then $\mathcal{A}$ is completely determined by its K-theory, up to isomorphism.
\end{thm}

In \cite{Eva2008} it is shown that, given a row-finite $k$-rank graph $\Lambda$ with no sources, the $\cst$-algebra $\cst(\Lambda)$ is separable, nuclear, unital, and satisfies the Universal Coefficient Theorem. Furthermore, we have shown in Proposition \ref{simple_purely-inf} that, given a complete bipartite graph $\kappa = \kappa(\alpha, \beta)$ with $\alpha, \beta \geq 3$, the $\cst$-algebra $\cst(\kappa)$ associated to its $2$-rank graph is simple and purely-infinite. Hence we can conclude:

\begin{corl}\label{corl_classification}
	Consider the complete bipartite graph $\kappa = \kappa(\alpha,\beta)$ for $\alpha,\beta \geq 3$, with corresponding $2$-rank graph $\Lambda(\kappa)$. Then the isomorphism class of the associated $\cst$-algebra $\cst(\Lambda(\kappa))$ is completely determined by the K-groups $K_0 (\cst(\kappa)) = K_1 (\cst(\kappa))$ and the position of the class of the identity in $K_0(\cst(\kappa))$.\qed
\end{corl}

\section{Unpointed tiles} \label{S_unpointed}

There is an alternative way we could have defined the adjacency matrices above, which will lead to a different $2$-rank graph structure.

Define an \textbf{unpointed tile system} $(G,U,V,\mathcal{S}')$ in the same way as Definition \ref{tile_system}, but replacing $\mathcal{S}=\mathcal{S}(G)$ with the set of unpointed tiles $\mathcal{S}'=\mathcal{S}'(G)$. We will see that analogues of the results in Section \ref{S_cst} also hold for unpointed tile systems.

\begin{defn}\label{adjacency_matrices_unpointed}
	Let $(G, U, V, \mathcal{S}')$ be an unpointed tile system, and let $A', B' \in \mathcal{S}'$ be unpointed tiles, that is, equivalence classes of some respective pointed tiles $A, B \in \mathcal{S}$ (c.f. Definition \ref{def_tiles}). Recall the matrices $M_1,M_2$ from Definition \ref{adjacency_matrices}. We define functions $M'_1,M'_2:\mathcal{S}' \times \mathcal{S}' \rightarrow \lbrace 0,1\rbrace$ as follows:
	\[
	\begin{array}{r}
	M'_1(A',B') = \left\{ \begin{array}{l} 
	1 \quad \text{if $M_1(A_\bullet ,B_\bullet) = 1$ for some $A_\bullet \sim A$, $B_\bullet \sim B$,} \\
	0 \quad \text{otherwise,} 
	\end{array}
	\right.\\
	\text{}\\
	M'_2(A',B') = \left\{ \begin{array}{l} 
	1 \quad \text{if $M_2(A_\bullet ,B_\bullet) = 1$ for some $A_\bullet \sim A$, $B_\bullet \sim B$,} \\
	0 \quad \text{otherwise.}
	\end{array}
	\right.
	\end{array}
	\]
	We define adjacency matrices $M'_1$, $M'_2$ accordingly.
\end{defn}

\begin{prop}\label{tile_complex_uce_unpointed}
	Consider the complete bipartite graph $\kappa = \kappa(\alpha,\beta)$ on $\alpha \geq 2$ white and $\beta \geq 2$ black vertices, and let $(\kappa, U, V, \mathcal{S}'(\kappa))$ be an unpointed tile system. Then the corresponding adjacency matrices $M'_1$ and $M'_2$ commute, and $(\kappa, U, V, \mathcal{S}'(\kappa))$ satisfies the UCE Property.
	
	Hence $(\kappa,U,V,\mathcal{S}'(\kappa))$ has a $2$-rank graph structure.
\end{prop}

\begin{proof}
	Given two unpointed tiles $A',B' \in \mathcal{S}'(\kappa)$, consider their respective sets of pointed tiles $\mathcal{A},\mathcal{B} \in \mathcal{S}(\kappa)$ as defined in Definition \ref{adjacency_matrices_unpointed}. Notice that $M'_1(A',B') = 1$ if and only if, for \textit{every} $A_\bullet \in \mathcal{A}$, we can find some $B_\bullet \in \mathcal{B}$ such that $M_1 (A_\bullet,B_\bullet) = 1$. The same is true for $M'_2$. Write $A' = \big( u_i^1, v_j^1, u_i^2, v_j^2 \big)$, and define sets
	\[
	X_A := \big\lbrace T \in \mathcal{S}'(\kappa) \mid M'_1(A,T) = 1\big\rbrace, \quad Y_A := \big\lbrace T \in \mathcal{S}'(\kappa) \mid M'_2(A,T) = 1\big\rbrace.
	\]
	Then $X_A$ contains precisely those tiles of the form $\big( u_k^1, v_j^1, u_k^2, v_j^2\big)$, where $k \neq i$, and $Y_A$ only those of the form $\big( u_i^1, v_l^1, u_i^2, v_l^2\big)$, where $l \neq j$. The proof then proceeds in a similar fashion to that of Proposition \ref{tile_complex_uce}, and the $2$-rank graph structure follows immediately from \cite[\S 6]{KumPas2000} as in Theorem \ref{tile_2rank}. 
\end{proof}

We write $\Lambda'(\kappa)$ for the $2$-rank graph induced from the adjacency matrices $M'_1$ and $M'_2$. It is not difficult to verify that $\Lambda'(\kappa)$ is row-finite, with finite vertex set and no sources. Hence we can apply Evans' Theorem \ref{evans}, and we derive the following result:

\begin{thm}[K-groups for unpointed tile systems]\label{unpointed}
	Let $a,b \geq 0$, and let $\kappa (a+2,b+2)$ be the complete bipartite graph on $a+2$ white and $b+2$ black vertices. Again, without loss of generality, we can assume that $a\leq b$. Write $\cst(\kappa) := \cst(\Lambda'(\kappa))$. Then, for $\epsilon = 0,1$:
	\begin{enumerate}[label=(\roman*)]
		\item If $a=b=0$, then $K_\epsilon(\cst(\kappa(a+2,b+2))=K_\epsilon(\cst(\kappa(2,2)) \cong \mathbb{Z}^2$.
		\item If $a=0$ and $b \geq 1$, then
		\[
		K_\epsilon(\cst(\kappa(a+2,b+2))) \cong ( \mathbb{Z}/2)^b\oplus ( \mathbb{Z}/(2b)).
		\]
		\item If $a,b \geq 1$, then 
		\[
		K_\epsilon(\cst(\kappa(a+2,b+2))) \cong ( \mathbb{Z}/2)^{(a+1)(b+1)-1}\oplus ( \mathbb{Z}/2g),
		\]
		where $g := \gcd(a, b)$.
	\end{enumerate}
\end{thm}

\begin{proof}
	Again, we start by proving (iii), as the first two cases follow. Write $\alpha := a+2$, $\beta := b+2$, and let $\alpha, \beta \geq 3$. For $1 \leq i \leq \alpha$, $1 \leq j \leq \beta$, write $A_{ij}'$ for the unpointed tile $\big( u_i^1, v_j^1, u_i^2, v_j^2 \big) \in \mathcal{S}'(\kappa)$. Then
	\begin{multline}\label{group2}
	\coker = \coker\big(\mathbf{1}-(M_1')^T, \mathbf{1}-(M_2')^T \big) = \Bigg\langle A_{ij}' \in \mathcal{S}'(\kappa) \Biggm\vert \\
	A_{ij}' = \sum_{T' \in \mathcal{S}'(\kappa)} M_1'(A_{ij}',T') \cdot T' = \sum_{T' \in \mathcal{S}'(\kappa)} M_2'(A_{ij}',T')\cdot T' \Bigg\rangle .
	\end{multline}
	Fix $p \in \lbrace 1,\ldots , \alpha \rbrace$, $q \in \lbrace 1,\ldots , \beta\rbrace$, and notice that:
	\begin{itemize}
		\item $M'_1(A_{pq}', T') = 1$ if and only if $T' = A_{iq}'$, for some $i \neq p$,
		\item $M'_2(A_{pq}', T') = 1$ if and only if $T' = A_{pj}'$, for some $j \neq q$.
	\end{itemize}
	Hence the relations of (\ref{group2}) are given by $A_{pq}' = \sum_{i \neq p} A_{iq}' = \sum_{j \neq q} A_{pj}'$. Define
	\[
	J_{pq} := \Bigg(\sum_{i=2}^{\alpha} A_{iq}'\Bigg) - A_{pq}' \quad \text{and}\quad I_{pq} := \Bigg(\sum_{j=2}^{\beta} A_{pj}'\Bigg) - A_{pq}',
	\]
	for $p, q \geq 2$. Then
	\begin{align*}
	2J_{pq} 	&= 2\Bigg(\sum_{i=2}^{\alpha} A_{iq}'\Bigg) - 2A_{pq}' \\
	&= 2(A_{2q}' + \cdots + A_{\alpha q}' - A_{pq}') + A_{1q}' - A_{1q}' \\
	&= (A_{1q}' + A_{2q}' + \cdots + A_{\alpha q}' - A_{pq}') + (-A_{1q}' + A_{2q}' + \cdots + A_{\alpha q}') - A_{pq}' \\
	&= A_{pq}' + 0 - A_{pq}' = 0,
	\end{align*}
	and similarly $2I_{pq} = 0$. Now, $J_{pq} = 0$ or $I_{pq} = 0$ only if $A_{pq}' = A_{1q}'$ or $A_{pq}' = A_{p1}'$ respectively. But since $\alpha, \beta \geq 3$, these equivalences are not relations of (\ref{group2}), and so $\ord(J_{pq}) = \ord(I_{pq}) = 2$. Notice that we can write each $A_{1q}'$ and $A_{p1}'$ in terms of the other $A_{ij}'$, for $p,q \geq 2$; hence we can remove these from the list of generators by a sequence of Tietze transformations.
	
	Also notice that we can write $A_{2q}' = J_{2q} - \sum_{i=3}^\alpha A_{iq}'$. Proceeding inductively, we can write each $A_{pq}'$ in terms of the $J_{iq}$ and the $A_{iq}'$ for $i > p$. Similarly, we can express each $A_{pq}'$ in terms of the $I_{pj}$ and the $A_{pj}'$ for $j > q$. Hence we can rewrite the generators of $\coker$ as $A_{11}'$, $I_{pq}$, $J_{pq}$, for $p,q \geq 2$.
	
	But $A_{11}' = -(A_{p1}'+J_{p1}) = -(A_{1q}'+I_{1q})$ for all $p,q \geq 2$, so 
	\[
	(\alpha - 2)A_{11}' = -\sum_{i=3}^{\alpha} (A_{i1}' + J_{i1})
	= -\Bigg( J_{21} + \sum_{i=3}^\alpha J_{i1}\Bigg ),
	\]
	and so $2(\alpha - 2)A_{11}' = 0$. Similarly, we find that $2(\beta - 2)A_{11}' = 0$, and hence that $2gA_{11}' = 0$, where $g:= \gcd(\alpha-2,\beta-2)$.
	
	Observe that, since $I_{pq}$ is defined in terms of the $A_{pj}'$, and each $A_{pj}'$ can be written in terms of the $J_{ij}$, we can remove the $I_{pq}$ from the list of generators of $\coker$. Finally, we can rewrite (\ref{group2}) as
	\[
	\coker = \langle J_{2q}, J_{p2}, J_{pq}, A_{11}' \mid 2J_{2q} = 2J_{p2} = 2J_{pq} = 2gA_{11}' = 0, \text{ for } 3 \leq p \leq \alpha, 3 \leq q \leq \beta \rangle,
	\]
	and after substituting $a = \alpha - 2$, $b = \beta - 2$, this gives a presentation for $(\mathbb{Z}/2)^{(a+1)(b+1)-1}\oplus (\mathbb{Z}/2g)$; since there is no torsion-free part, this proves (iii).
	
	If $\alpha = 2$, then $A_{1q}' = A_{2q}'$ for all $1 \leq q \leq \beta$, so we can write
	\[
	\coker = \Bigg\langle A_{1q}' \Biggm\vert A_{1q}' = \sum_{j\neq q} A_{1j}', \text{ for } 1 \leq q \leq \beta \Bigg\rangle .
	\]
	We adjust the proof above accordingly to obtain the result of (ii). Finally, in case (i) where $\alpha = \beta = 2$, we have $A_{11}' = A_{12}' = A_{21}' = A_{22}'$ with no further relations, such that $\coker = \langle A_{11}' \rangle \cong \mathbb{Z}$, and the result follows from Theorem \ref{evans}.	
\end{proof}

\begin{thm}\label{unpointed_identity}
	Let $\alpha, \beta \geq 3$, let $\kappa = \kappa (\alpha, \beta)$ be the complete bipartite graph on $\alpha$ white and $\beta$ black vertices, and write $g := \gcd(\alpha - 2, \beta - 2)$. Then the order of the class of the identity $[\mathbf{1}]$ in $K_0(\cst(\Lambda'(\kappa)))$ is equal to $g$ if $g$ is odd, and $g/2$ if $g$ is even.
\end{thm}

\begin{proof}
	Consider the notation used in the proof of Theorem \ref{unpointed}. 
	As with Theorem \ref{pointed_identity}, we know that the order of $[\mathbf{1}]$ in $K_0(\cst(\kappa))$ is equal to the order of the sum of all tiles $A_{ij}'$. We write $\Sigma$ for this sum.
	
	We have that $A_{pq}' = \sum_{i \neq p} A_{iq}' = \sum_{j \neq q} A_{pj}'$, and so $\Sigma = (\alpha - 1)\Sigma = (\beta - 1)\Sigma$. From this, it follows that $g\Sigma = 0$. We also have $A_{pq}' = \sum_{i \neq p} \sum_{j \neq q} A_{ij}'$, so that
	\[
	\Sigma = A_{pq}' + \sum_{i \neq p} A_{iq}' + \sum_{j \neq q} A_{pj}' + \sum_{i \neq p} \sum_{j \neq q} A_{ij}'
	= 4A_{pq}',
	\]
	for any fixed $p,q$. But $2gA_{pq}' = 0$, and so if $g=2h$ for some integer $h$, then $h\Sigma = 4hA_{pq}' = 0$, and we are done.
\end{proof}

The proof of the following relies on identical results to those in Section \ref{S_aperiodicity}.

\begin{prop}\label{unpointed_classification}
	Consider the complete bipartite graph $\kappa = \kappa(\alpha,\beta)$ for $\alpha,\beta \geq 3$, and the associated $2$-rank graph $\Lambda'(\kappa)$. Then the isomorphism class of the universal $\cst$-algebra $\cst(\Lambda'(\kappa))$ is completely determined by its K-theory and the position of the class of the identity in $K_0(\cst(\Lambda'(\kappa)))$.\qed
\end{prop}

\section{The homology of a tile complex} \label{S_homology}

\begin{thm}\label{homology}
	Let $\kappa = \kappa (\alpha,\beta)$ be the complete bipartite graph on $\alpha \geq 2$ white and $\beta \geq 2$ black vertices, let $(\kappa, U, V, \mathcal{S}'(\kappa))$ be an unpointed tile system, and let $TC(\kappa)$ be its associated tile complex. Then the homology groups of $TC(\kappa)$ are given by:
	\[
	H_n(TC(\kappa)) \cong \begin{cases}
	0 & \text{ for }n = 0, \\
	\mathbb{Z}^{\alpha + \beta - 2} & \text{ for } n = 1, \\
	\mathbb{Z}^{(\alpha - 1)(\beta - 1)} & \text{ for } n = 2, \\
	0 & \text{ for } n \geq 3.
	\end{cases}
	\]
\end{thm}

\begin{proof}
	As $TC(\kappa)$ is a path-connected, $2$-dimensional CW-complex by construction, clearly $H_n(TC(\kappa)) \cong 0$ for $n=0$ and $n \geq 3$. 
	
	The proof uses as its basis that of \cite[Proposition 3]{NorThoVdo2018}. The boundary of each square in $TC(\kappa)$ is given by an element of $\mathcal{S}'(\kappa)$; write these elements as $\big(u_i^1,v_j^1,u_i^2,v_j^2\big)$. By construction, $TC(\kappa)$ has four vertices: each one the origin of all directed edges labelled $u_i^1$, $v_j^1$, $u_i^2$, and $v_j^2$ respectively. Each tile is homotopy equivalent to a point; pick tile $\big(u_1^1,v_1^1,u_1^2,v_1^2\big)$ and contract it, thereby identifying the four vertices. Call the resulting tile complex $TC_1(\kappa)$.
	
	This is a $2$-dimensional CW-complex whose edges are loops, and whose $2$-cells comprise:
	\begin{itemize}
		\item $(\alpha-1)(\beta-1)$-many unpointed tiles $A_{ij}'=\big(u_i^1,v_j^1,u_i^2,v_j^2\big)$,
		\item $(\alpha-1)$-many $2$-gons $X_i'$ with boundaries described analogously by $\big(u_i^1,u_i^2\big)$,
		\item $(\beta-1)$-many $2$-gons $Y_j'$ with boundaries described by $\big(v_j^1,v_j^2\big)$,
	\end{itemize}
	for $2 \leq i \leq \alpha$, $2 \leq j \leq \beta$. Consider the chain complex associated to $TC_1(\kappa)$:
	\[
	\cdots \longrightarrow C_3 \overset{\partial_3}{\longrightarrow} C_2 \overset{\partial_2}{\longrightarrow} C_1 \overset{\partial_1}{\longrightarrow} C_0 \overset{\partial_0}{\longrightarrow} 0.
	\]
	Since $TC_1(\kappa)$ is $2$-dimensional and has one vertex, this boils down to
	\[
	0 \overset{0}{\longrightarrow} C_2 \overset{\partial_2}{\longrightarrow} C_1 \overset{0}{\longrightarrow} 0,
	\]
	and so $H_1(TC_1(\kappa)) \cong C_1 / \im(\partial_2)$, and $H_2(TC_1(\kappa)) \cong \ker(\partial_2)$. We have $\partial_2(A_{ij}') = u_i^1 + v_j^1 + u_i^2 + v_j^2$, $\partial_2(X_i') = u_i^1 + u_i^2$, and $\partial_2(Y_j') = v_j^1 + v_j^2$. Clearly $\ker(\partial_2)$ is generated by $\lbrace A_{ij}' - X_i' - Y_j' \mid 2 \leq i \leq \alpha, 2 \leq j \leq \beta \rbrace$, such that $\ker(\partial_2) \cong \mathbb{Z}^{(\alpha - 1 )( \beta - 1)}$.
	
	Similarly, we have an Abelian group presentation for $H_1(TC_1(\kappa))$ as follows:
		\begin{multline*}
		H_1(TC_1(\kappa)) \cong \big\langle u_i^1,v_j^1,u_i^2,v_j^2 \bigm\vert u_i^1 + v_j^1 + u_i^2 + v_j^2 = u_i^1 + u_i^2 = v_j^1 + v_j^2 = 0,\\
		\text{for } 2 \leq i \leq \alpha, 2 \leq j \leq \beta \big\rangle,
		\end{multline*}
	which, after substituting $u_i^2 = -u_i^1$ and $v_j^2 = -v_j^1$, gives 
	\[
	H_1(TC_1(\kappa)) \cong \big\langle u_i^1,v_j^1, \text{ for } 2 \leq i \leq \alpha, 2 \leq j \leq \beta \big\rangle.
	\]
	This is a presentation for $\mathbb{Z}^{\alpha + \beta - 2}$, and since $TC_1(\kappa)$ is homotopy equivalent to $TC(\kappa)$, we are done.
\end{proof}

\section{Pointed and unpointed $2t$-gon systems}\label{S_polygons}

In this section we suggest generalisations of the methods above for constructing $\cst$-algebras associated to $2t$-gon systems, both for even and arbitrary $t \geq 1$.

When $t=2$, we have an innate idea of what it means for two $2t$-gons to be `stackable:' functions we called horizontal and vertical adjacency in Definition \ref{adjacency_matrices}. We extend this notion to all even $t \geq 2$ in as natural a way possible.

Definition \ref{pointed_polygons} directly generalises the definitions at the beginning of Section \ref{S_tile}.

\begin{defn}\label{pointed_polygons}
	Let $G$ be a connected bipartite graph on $\alpha$ white and $\beta$ black vertices. Let $U$, $V$ be sets with $|U| = 2t\alpha$, $|V| = 2t\beta$, and which are gifted with fixed-point-free involutions $u \mapsto \bar{u}$, $v \mapsto \bar{v}$ respectively. Construct the $2t$-polyhedron $P(G)$ from Theorem \ref{vdovina}, which has $G$ as its link at each vertex, using $U$ and $V$, and write $\mathcal{S}'(G) := \lbrace A_e \mid e \in E(G)\rbrace$ for the set of $2t$-gons which comprise $P(G)$. We call elements of $\mathcal{S}_t'(G)$ \textbf{unpointed $2t$-gons}, and denote them by $A_e = ( x_1,y_1,\ldots ,x_t,y_t)$. 
	
	Analogously to Section \ref{S_tile}, we write $[ x_1,y_1,\ldots , x_t,y_t ]$ for a \textbf{pointed $2t$-gon}, that is, a $2t$-gon labelled anticlockwise and starting from a distinguished basepoint by the sequence $x_1, y_1, \ldots , x_t, y_t$, for some $x_i \in U$, $y_i \in V$. Write $\mathcal{S}_t=\mathcal{S}_t(G)$ for the set of $2t\alpha\beta$ pointed $2t$-gons. We call the tuple $(G,U,V,\mathcal{S}_t)$ a $2t$\textbf{-gon system}. Similarly we call a tuple $(G,U,V,\mathcal{S}_t')$ an \textbf{unpointed $2t$-gon system}.
\end{defn}

Consider the adjacency matrices $M_1$, $M_2$ from Definition \ref{adjacency_matrices}. We can view two pointed tiles ($4$-gons) $A=[x_1,y_1,x_2,y_2]$ and $B$ as being horizontally-adjacent, that is, $M_1(A,B)=1$) if and only if, after reflecting $A$ through an axis connecting the midpoints of $x_1$ and $x_2$, and then replacing $x_1$, $x_2$ by some $x_1' \neq x_1$, $x_2' \neq x_2$ respectively, we can obtain $B$. Likewise, if and only if we can obtain $B$ by reflecting $A$ through an axis joining the midpoints of the $y$ edges, and then changing the labels of those edges, do we say that $A$ and $B$ are vertically-adjacent.

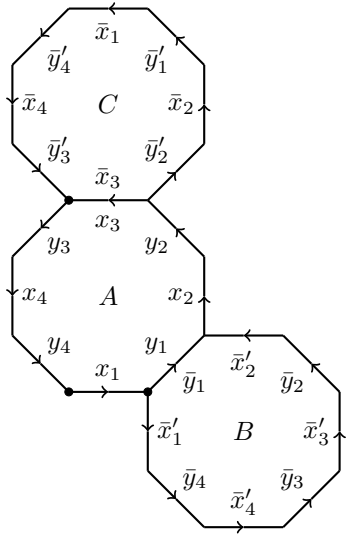
\begin{figure}[h]
	\begin{center}
		\begin{tikzpicture}
		\begin{scope}[scale=0.75]
		
		\draw[->,thick] (0,0) -- (0.7,0);
		\draw[thick] (0.7,0) -- (1.4,0);
		\draw[->,thick] (1.4,0) -- (1.9,0.5);
		\draw[thick] (1.9,0.5) -- (2.4,1);
		\draw[->,thick] (2.4,1) -- (2.4,1.7);
		\draw[thick] (2.4,1.7) -- (2.4,2.4);
		\draw[->,thick] (2.4,2.4) -- (1.9,2.9);
		\draw[thick] (1.9,2.9) -- (1.4,3.4);
		\draw[->,thick] (1.4,3.4) -- (0.7,3.4);
		\draw[thick] (0.7,3.4) -- (0,3.4);
		\draw[->,thick] (0,3.4) -- (-0.5,2.9);
		\draw[thick] (-0.5,2.9) -- (-1,2.4);
		\draw[->,thick] (-1,2.4) -- (-1,1.7);
		\draw[thick] (-1,1.7) -- (-1,1);
		\draw[->,thick] (-1,1) -- (-0.5,0.5);
		\draw[thick] (-0.5,0.5) -- (0,0);
		\filldraw (0,0) circle (2pt);
		
		\begin{scope}[xshift=2.4cm,yshift=-2.4cm]
		\draw[->,thick] (0,0) -- (0.7,0);
		\draw[thick] (0.7,0) -- (1.4,0);
		\draw[->,thick] (1.4,0) -- (1.9,0.5);
		\draw[thick] (1.9,0.5) -- (2.4,1);
		\draw[->,thick] (2.4,1) -- (2.4,1.7);
		\draw[thick] (2.4,1.7) -- (2.4,2.4);
		\draw[->,thick] (2.4,2.4) -- (1.9,2.9);
		\draw[thick] (1.9,2.9) -- (1.4,3.4);
		\draw[->,thick] (1.4,3.4) -- (0.7,3.4);
		\draw[thick] (0.7,3.4) -- (0,3.4);
		
		\draw[->,thick] (-1,2.4) -- (-1,1.7);
		\draw[thick] (-1,1.7) -- (-1,1);
		\draw[->,thick] (-1,1) -- (-0.5,0.5);
		\draw[thick] (-0.5,0.5) -- (0,0);
		\filldraw (-1,2.4) circle (2pt);
		
		\draw[] (0.7,1.7) node{$B$};
		\draw[anchor=south,shift={(0,0)}] (0.7,0.05) node{$\bar{x}_4'$};
		\draw[anchor=south east,shift={(0,0)}] (1.95,0.5) node{$\bar{y}_3$};
		\draw[anchor=east,shift={(0,0)}] (2.4,1.7) node{$\bar{x}_3'$};
		\draw[anchor=north east,shift={(0,0)}] (1.95,2.9) node{$\bar{y}_2$};
		\draw[anchor=north,shift={(0,0)}] (0.7,3.3) node{$\bar{x}_2'$};
		\draw[anchor=north west,shift={(0,0)}] (-0.55,2.9) node{$\bar{y}_1$};
		\draw[anchor=west,shift={(0,0)}] (-1,1.7) node{$\bar{x}_1'$};
		\draw[anchor=south west,shift={(0,0)}] (-0.55,0.5) node{$\bar{y}_4$};
		\end{scope}
		
		\begin{scope}[yshift=3.4cm]
		\draw[->,thick] (1.4,0) -- (1.9,0.5);
		\draw[thick] (1.9,0.5) -- (2.4,1);
		\draw[->,thick] (2.4,1) -- (2.4,1.7);
		\draw[thick] (2.4,1.7) -- (2.4,2.4);
		\draw[->,thick] (2.4,2.4) -- (1.9,2.9);
		\draw[thick] (1.9,2.9) -- (1.4,3.4);
		\draw[->,thick] (1.4,3.4) -- (0.7,3.4);
		\draw[thick] (0.7,3.4) -- (0,3.4);
		\draw[->,thick] (0,3.4) -- (-0.5,2.9);
		\draw[thick] (-0.5,2.9) -- (-1,2.4);
		\draw[->,thick] (-1,2.4) -- (-1,1.7);
		\draw[thick] (-1,1.7) -- (-1,1);
		\draw[->,thick] (-1,1) -- (-0.5,0.5);
		\draw[thick] (-0.5,0.5) -- (0,0);
		\filldraw (0,0) circle (2pt);
		
		\draw[] (0.7,1.7) node{$C$};
		\draw[anchor=south,shift={(0,0)}] (0.7,0.05) node{$\bar{x}_3$};
		\draw[anchor=south east,shift={(0,0)}] (1.95,0.5) node{$\bar{y}_2'$};
		\draw[anchor=east,shift={(0,0)}] (2.4,1.7) node{$\bar{x}_2$};
		\draw[anchor=north east,shift={(0,0)}] (1.95,2.9) node{$\bar{y}_1'$};
		\draw[anchor=north,shift={(0,0)}] (0.7,3.3) node{$\bar{x}_1$};
		\draw[anchor=north west,shift={(0,0)}] (-0.55,2.9) node{$\bar{y}_4'$};
		\draw[anchor=west,shift={(0,0)}] (-1,1.7) node{$\bar{x}_4$};
		\draw[anchor=south west,shift={(0,0)}] (-0.55,0.5) node{$\bar{y}_3'$};
		\end{scope}
		
		\draw[] (0.7,1.7) node{$A$};
		
		\draw[anchor=south,shift={(0,0)}] (0.7,0.05) node{$x_1$};
		\draw[anchor=south east,shift={(0,0)}] (1.95,0.5) node{$y_1$};
		\draw[anchor=east,shift={(0,0)}] (2.4,1.7) node{$x_2$};
		\draw[anchor=north east,shift={(0,0)}] (1.95,2.9) node{$y_2$};
		\draw[anchor=north,shift={(0,0)}] (0.7,3.3) node{$x_3$};
		\draw[anchor=north west,shift={(0,0)}] (-0.55,2.9) node{$y_3$};
		\draw[anchor=west,shift={(0,0)}] (-1,1.7) node{$x_4$};
		\draw[anchor=south west,shift={(0,0)}] (-0.55,0.5) node{$y_4$};
		\end{scope}
		\end{tikzpicture}
	\end{center}
	\caption{$U$-, $V$-adjacency: Consider the pointed octagons $A = [ x_1,y_1,\ldots ,x_4,y_4]$, $B = [ \bar{x}_1', \bar{y}_4, \ldots,\bar{x}_2', \bar{y}_1]$, and $C = [\bar{x}_3, \bar{y}_2', \ldots ,\bar{x}_4, \bar{y}_3']$ in $\mathcal{S}_4$. We say that $A$ and $B$ are $V$-adjacent, and $A$ and $C$ are $U$-adjacent. There is a unique octagon $D = [x_3',y_3', \ldots , x_2', y_2']$ which is both $U$-adjacent to $B$ and $V$-adjacent to $C$.}\label{fig_UV-adjacency}
\end{figure}

\begin{defn}\label{UV-adjacency}
	Let $t$ be an even integer, let $(G,U,V,\mathcal{S}_t)$ be a $2t$-gon system, and let $A = [ x_1,y_1,\ldots ,x_t,y_t] \in \mathcal{S}_t$ be a pointed $2t$-gon.
	
	Reflect $A$ through an axis joining the midpoints of sides labelled $x_1$ and $x_{(t/2)+1}$ to obtain a new pointed $2t$-gon $[ \bar{x}_1, \bar{y}_t, \bar{x}_t, \bar{y}_{t-1}, \ldots , \bar{x}_2, \bar{y}_1]$. We say that a pointed $2t$-gon $B \in \mathcal{S}_t$ is $V$\textbf{-adjacent} to $A$ if $B = \big[ \bar{x}_1', \bar{y}_t, \bar{x}_t', \bar{y}_{t-1}, \ldots , \bar{x}_2', \bar{y}_1 \big]$, for some $x_i' \neq x_i$.
	
	Similarly, reflect $A$ such that $x_1 \mapsto \bar{x}_{(t/2)+1}$; we obtain a new pointed $2t$-gon
		\begin{equation}\label{U-adj}
		\big[ \bar{x}_{(t/2)+1}, \bar{y}_{t/2}, \bar{x}_{t/2}, \ldots , \bar{y}_1, \bar{x}_1, \bar{y}_t, \bar{x}_t, \ldots , \bar{x}_{(t/2)+2}, \bar{y}_{(t/2)+1} \big].
		\end{equation}
	We say that a pointed $2t$-gon $B \in \mathcal{S}_t$ is $U$\textbf{-adjacent} to $A$ if $B$ is of the form (\ref{U-adj}), but with all elements $y_i$ replaced with some $y_i' \neq y_i$ (Figure \ref{fig_UV-adjacency}).
	
	We define the $U$\textbf{-} and $V$\textbf{-adjacency matrices}, $M_U$ and $M_V$ respectively, to be the $2t\alpha\beta \times 2t\alpha\beta$ matrices with $AB$-th entry $1$ if $A$ and $B$ are $U$-adjacent (resp. $V$-adjacent), and $0$ otherwise.
\end{defn}

\begin{prop}\label{polygon_uce}
	Let $t$ be even, and $(\kappa, U, V, \mathcal{S}_t(\kappa))$ be a $2t$-gon system with adjacency matrices $M_U$, $M_V$. Then these matrices commute, and $(\kappa, U, V, \mathcal{S}_t(\kappa))$ satisfies the UCE Property.
	
	Hence $(\kappa,U,V,\mathcal{S}_t(\kappa))$ has a $2$-rank graph structure.
\end{prop}

\begin{proof}
	Consider the pointed $2t$-gon $A = \big[ u_i^1,v_j^1,\ldots ,u_i^t,v_j^t\big] \in \mathcal{S}_t(\kappa)$; those $2t$-gons corresponding to its reflections and rotations are treated similarly. Then a pointed $2t$-gon $B$ is $V$-adjacent to $A$ if and only if $B = \big[ \bar{u}_k^1, \bar{v}_j^t, \ldots , \bar{u}_k^2, \bar{v}_j^1 \big]$, for some $k \neq i$. Suppose $B$ is such a $2t$-gon $V$-adjacent to $A$; then a pointed $2t$-gon $D$ is $U$-adjacent to $B$ if and only if
		\begin{equation}\label{D}
		D = \Big[ u_k^{(t/2)+1}, v_l^{(t/2)+1}, \ldots , u_k^t, v_l^t, u_k^1, v_k^1, \ldots , u_k^{t/2}, v_l^{t/2} \Big],
		\end{equation}
	for some $l \neq j$. Likewise, $C$ is $U$-adjacent to $A$ if and only if 
		\[
		C = \Big[ \bar{u}_i^{(t/2)+1}, \bar{v}_l^{t/2}, \ldots , \bar{u}_i^1, \bar{v}_l^t, \ldots, \bar{u}_i^{(t/2)+2}, \bar{v}_l^{(t/2)+1} \Big],
		\]
	for some $l \neq j$. Clearly if $C$ is such a $2t$-gon, then $D$ is $V$-adjacent to $C$ if and only if it is of the form (\ref{D}). Exactly one such $D$ exists in $\mathcal{S}_t(\kappa)$, hence $M_U$ and $M_V$ commute. Then $(\kappa, U, V, \mathcal{S}_t(\kappa))$ has the UCE Property, and the $2$-rank graph structure follows from \cite[\S 6]{KumPas2000}.
\end{proof}

Recall the $2$-rank graph $\Lambda(\kappa)$ induced from a tile system and its adjacency matrices $M_1$, $M_2$ in Section \ref{S_cst}, and recall its associated universal $\cst$-algebra $\cst(\Lambda)$ from Definition \ref{graph_algebra}. Similarly, we write $\Lambda_t(\kappa)$ for the $2$-rank graph induced from the $U$- and $V$-adjacency matrices $M_U$ and $M_V$, and observe that $\Lambda_t(\kappa)$ is row-finite, with finite vertex set and no sources. Hence from Evans' Theorem \ref{evans}, we can deduce:

\begin{thm}[K-groups for pointed $2t$-gon systems, $t$ even] \label{pointed_polygon}
	Let $\alpha, \beta \geq 2$, let $t \geq 2$ be even, and let $\kappa = \kappa (\alpha , \beta)$ be the complete bipartite graph on $\alpha$ white and $\beta$ black vertices. Then
		\[
		K_\epsilon(\cst(\Lambda_t(\kappa))) \cong (K_\epsilon(\cst(\Lambda(\kappa))))^{t/2},
		\]
	for $\epsilon = 0,1$.
\end{thm}

\begin{proof}
	Fix $t$ and assume without loss of generality that $\alpha \leq \beta$. Analogously to in the proof of Theorem \ref{pointed}, we denote the pointed $2t$-gons in $\mathcal{S}_t(\kappa)$ as follows:
		\begin{itemize}
		\item $(A_r)_{ij} := \big[u_i^r, v_j^r, \ldots , u_i^t, v_j^t, u_i^1, v_j^1, \ldots , u_i^{r-1}, v_j^{r-1} \big]$,
		\item $(B_r)_{ij} := \big[ \bar{u}_i^r, \bar{v}_j^{r-1}, \ldots , \bar{u}_i^1, \bar{v}_j^t, \ldots , \bar{u}_i^{r+1}, \bar{v}_j^r \big]$,
		\item $(C_r)_{ij} := \Big[ \bar{u}_i^{(t/2)+r}, \bar{v}_j^{(t/2)+r-1}, \ldots , \bar{u}_i^1, \bar{v}_j^t, \ldots, \bar{u}_i^{(t/2)+r+1}, \bar{v}_j^{(t/2)+r} \Big]$,
		\item $(D_r)_{ij} := \Big[ u_i^{(t/2)+r}, v_j^{(t/2)+r}, \ldots , u_i^t, v_j^t, u_i^1, v_j^1, \ldots , u_i^{(t/2)+r-1}, v_j^{(t/2)+r-1} \Big]$,
		\end{itemize}
	for $1 \leq i \leq \alpha$, $1 \leq j \leq \beta$, $1 \leq r \leq t/2$, and with addition in superscript indices defined modulo $t$. Note that each $S \in \mathcal{S}_t(\kappa)$ takes one of the above forms. Then
		\begin{align*}
		\coker \big( \mathbf{1}-M_U^T, \mathbf{1}-M_V^T\big) = \Bigg\langle
		(A_r)_{pq}&, (B_r)_{pq}, (C_r)_{pq}, (D_r)_{pq} \Biggm\vert \\ &(A_r)_{pq} = \sum_{i \neq p} (B_r)_{iq} = \sum_{j \neq q} (C_r)_{pj}, \\
		&(B_r)_{pq} = \sum_{i \neq p} (A_r)_{iq} = \sum_{j \neq q} (D_r)_{pj}, \\
		&(C_r)_{pq} = \sum_{i \neq p} (D_r)_{iq} = \sum_{j \neq q} (A_r)_{pj}, \\
		&(D_r)_{pq} = \sum_{i \neq p} (C_r)_{iq} = \sum_{j \neq q} (B_r)_{pj}, \\
		&\text{for }1 \leq p \leq \alpha, 1 \leq q \leq \beta, \text{ and }1 \leq r \leq t/2  \Bigg\rangle .
		\end{align*}
	But, comparing this to (\ref{group1}), we see this is precisely a presentation for the direct sum of $t/2$ copies of $\coker\big( I-M_1^T, I-M_2^T\big)$ as in Theorem \ref{pointed}, and the result follows.
\end{proof}

\begin{thm}\label{polygon_identity}
	Let $\alpha, \beta \geq 3$, let $t \geq 2$ be even, and let $\kappa = \kappa (\alpha, \beta)$ be the complete bipartite graph on $\alpha$ white and $\beta$ black vertices. Then the order of the class of the identity $[\mathbf{1}]$ in $K_0(\cst(\Lambda_t(\kappa)))$ is equal to $g := \gcd(\alpha - 2, \beta - 2)$.
	
	Furthermore, the isomorphism class of $\cst(\Lambda_t(\kappa))$ is completely determined by the K-groups in Theorem \ref{pointed_polygon} and the order of $[\mathbf{1}]$ in $K_0$.
\end{thm}

\begin{proof}
	The result follows from Theorems \ref{pointed_identity} and \ref{pointed_polygon}, and similar considerations to those in Section \ref{S_aperiodicity}.
\end{proof}

If we extend the concept of $U$- and $V$-adjacency from Definition \ref{UV-adjacency} in the obvious way, we can obtain a generalisation of Section \ref{S_unpointed} for unpointed $2t$-gon systems of complete bipartite graphs. Write $\Lambda_t'(\kappa)$ for the induced $2$-rank graph. We realise that the proof of Theorem \ref{unpointed} does not depend on the number of sides $2t$ of the $2t$-gons; hence nor do the K-groups associated to $\Lambda_t'(\kappa)$.

\begin{corl}[to Theorem \ref{unpointed}: K-groups for unpointed $2t$-gon systems]\label{unpointed_polygon}
	Let $\alpha,\beta \geq 2$, and let $\kappa = \kappa(\alpha,\beta)$ be the complete bipartite graph on $\alpha$ white and $\beta$ black vertices. Then
		\[
		K_\epsilon(\cst(\Lambda_t'(\kappa))) \cong K_\epsilon(\cst(\Lambda'(\kappa))),
		\]
	for $\epsilon = 0,1$, and all $t \geq 1$.\qed
\end{corl}

\begin{prop}\label{unpointed_polygon_identity}
	Let $\alpha, \beta \geq 3$, let $\kappa = \kappa (\alpha, \beta)$ be the complete bipartite graph on $\alpha$ white and $\beta$ black vertices, and write $g := \gcd(\alpha - 2, \beta - 2)$. Then for all $t \geq 1$, the order of the class of the identity $[\mathbf{1}]$ in $K_0(\cst(\Lambda_t'(\kappa)))$ is equal to $g$ if $g$ is odd, and $g/2$ if $g$ is even.
	
	Furthermore, the isomorphism class of $\cst(\Lambda_t'(\kappa))$ is completely determined by the K-groups in Corollary \ref{unpointed_polygon} and the order of $[\mathbf{1}]$ in $K_0$.\qed
\end{prop}

\begin{corl}[to Theorem \ref{homology}]\label{polygon_homology}
	Let $(\kappa, U, V, \mathcal{S}_t'(\kappa))$ be an unpointed $2t$-gon system, and let $P(\kappa)$ be its associated $2t$-polyhedron. Then the homology groups of $P(\kappa)$ do not depend on $t$, that is:
	\[
	H_n(P(\kappa)) \cong \begin{cases}
	0 & \text{ for }n = 0, \\
	\mathbb{Z}^{\alpha + \beta - 2} & \text{ for } n = 1, \\
	\mathbb{Z}^{(\alpha - 1)(\beta - 1)} & \text{ for } n = 2, \\
	0 & \text{ for } n \geq 3.
	\end{cases}
	\]\qed
\end{corl}

\subsection*{Questions on canonicality}

Corollary \ref{unpointed_polygon} gives us a collection of K-groups corresponding to systems of $2t$-gons with an arbitrary even number of sides $2t$, whereas in the pointed case, Theorem \ref{pointed_polygon} insists on $2t$ being divisible by four. This is due to how we define adjacency in each instance: in the $2t$-polyhedron $P(\kappa)$, each face is adjacent to every other, and since the number of faces is not dependent on $t$, nor are the $U$- and $V$-adjacency matrices in an unpointed $2t$-gon system.

Adjacency in the pointed case is more difficult to define canonically. When $t=2$, and we are dealing with tiles, there is an obvious pair of adjacency functions. We extended these in Definition \ref{UV-adjacency}, thinking of two $2t$-gons as adjacent if we can reflect one horizontally or vertically in order to obtain the form of the other. This works since horizontal and vertical reflections commute, and so the $2t$-gon system will satisfy the UCE Property. If $t$ is not even, then there are no two distinct reflections of $2t$-gons which commute, and preserve the structure of pointed $2t$-gons. We must pick the same two reflections for both adjacency functions, else some combination of rotations and identity transformations. None of these options is a direct extension of our horizontal and vertical adjacency functions from Section \ref{S_cst}, and so there is no natural choice.

We suggest that the following definitions of $U$- and $V$-adjacency for pointed $2t$-gons are the most intuitive $t \geq 3$, based on the idea that adjacent $2t$-gons should have opposite orientations. They do not, however, generalise the tile systems from Sections \ref{S_tile}--\ref{S_unpointed}, themselves being the most natural constructions when $t=2$. Hence, the previous constructions have been the main focus of this paper.

\begin{defn}\label{new_UV-adjacency}
	Let $t \geq 1$ be a fixed arbitrary integer, let $(G,U,V,\mathcal{S}_t)$ be a $2t$-gon system, and let $A = [ x_1,y_1,\ldots ,x_t,y_t] \in \mathcal{S}_t$ be a pointed $2t$-gon.
	
	A pointed $2t$-gon $B \in \mathcal{S}_t$ is $V^*$\textbf{-adjacent} to $A$ if and only if $B = [\bar{x}_1', \bar{y}_t, \ldots , \bar{x}_2', \bar{y}_1]$, for some $x_i' \neq x_i$.
	
	Similarly, we say that a pointed $2t$-gon $C \in \mathcal{S}_t$ is $U^*$\textbf{-adjacent} to $A$ if and only if $C = [ \bar{x}_1, \bar{y}_t', \ldots , \bar{x}_2, \bar{y}_1' ]$, for some $y_i' \neq y_i$. We define the $U^*$\textbf{-} and $V^*$\textbf{-adjacency matrices} $M_U^*$ and $M_V^*$ respectively, as above.
\end{defn}

The proof of the following is almost identical to that of Proposition \ref{polygon_uce}, together with Proposition \ref{tile_2rank}. From this, along with Theorem \ref{evans}, we can deduce Theorem \ref{new_polygons}

\begin{prop}
	Let $(\kappa, U, V, \mathcal{S}_t(\kappa))$ be a $2t$-gon system with adjacency matrices $M_U^*$, $M_V^*$. Then $(\kappa, U, V, \mathcal{S}_t(\kappa))$ induces a $2$-rank graph $\Lambda_t^*(\kappa)$, which is row-finite, with finite vertex set and no sources.\qed
\end{prop}

\begin{thm}[K-groups for pointed $2t$-gon systems, $t$ arbitrary] \label{new_polygons}
	Let $a,b \geq 0$, let $t \geq 1$, and let $\kappa = \kappa(a+2, b+2)$ be the complete bipartite graph on $a+2$ white and $b+2$ black vertices. Without loss of generality, we assume that $a \leq b$. Then, for $\epsilon = 0,1$:
		\begin{enumerate}[label=(\roman*)]
			\item If $a=b=0$, then $K_\epsilon(\cst (\Lambda_t^*(\kappa) )) \cong \mathbb{Z}^{4t}$.
			\item If $b \geq 1$ and $a,b$ are coprime, then $K_\epsilon (\cst (\Lambda_t^*(\kappa) )) \cong \mathbb{Z}^{2t(a+1)(b+1)}$.
			\item If $b \geq 1$ and $a,b$ are \emph{not} coprime, then
				\[
				K_\epsilon (\cst (\Lambda_t^*(\kappa) )) \cong \mathbb{Z}^{2t(a+1)(b+1)} \oplus (\mathbb{Z}/g)^t,
				\]
			where $g := \gcd(a,b)$.
		\end{enumerate}
\end{thm}

\begin{proof}
	The proof unsurprisingly follows the same lines as those of Theorems \ref{pointed}, \ref{unpointed}, and \ref{pointed_polygon}. Write $\alpha := a+2$, $\beta := b+2$, and let $\beta \geq 3$. We denote the pointed $2t$-gons in $\mathcal{S}_t(\kappa)$ as:
	\begin{itemize}
		\item $(A_r)_{ij} := \big[u_i^r, v_j^r, \ldots , u_i^t, v_j^t, u_i^1, v_j^1, \ldots , u_i^{r-1}, v_j^{r-1} \big]$,
		\item $(B_r)_{ij} := \big[ \bar{u}_i^r, \bar{v}_j^{r-1}, \ldots , \bar{u}_i^1, \bar{v}_j^t, \ldots , \bar{u}_i^{r+1}, \bar{v}_j^r \big]$,
	\end{itemize}
	for $1 \leq i \leq \alpha$, $1 \leq j \leq \beta$, $1 \leq r \leq t$, and with addition in superscript indices defined modulo $t$. Observe that each $S \in \mathcal{S}_t(\kappa)$ is either of the form $(A_r)_{ij}$ or $(B_r)_{ij}$. Then
		\begin{align*}
		\coker=\coker\big( \mathbf{1} - (M_U^*)^T, \mathbf{1} - (M_V^*)^T \big) = \Bigg\langle &(A_r)_{pq}, (B_r)_{pq} \Biggm \vert \\
		&(A_r)_{pq} = \sum_{i \neq p} (B_r)_{iq} = \sum_{j \neq q} (B_r)_{pj},\\
		&(B_r)_{pq} = \sum_{i \neq p} (A_r)_{iq} = \sum_{j \neq q} (A_r)_{pj},\\
		&\text{for } 1 \leq p \leq \alpha, 1 \leq q \leq \beta, \text{ and } 1 \leq r \leq t \Bigg\rangle.
		\end{align*}
	As in the proof of Theorem \ref{pointed}, define $(J_r)_q := \sum_{i=1}^\alpha (A_r)_{iq}$, and $(I_r)_p := \sum_{j=1}^\beta (A_r)_{pj}$. By means of a sequence of Tietze transformations, and using some observations from previous proofs, we see that the above presentation is equivalent to
		\begin{align*}
		\coker &= \Bigg\langle (A_r)_{pq} \Biggm \vert (A_r)_{pq} = \sum_{i \neq p} \sum_{k \neq i} (A_r)_{kq} = \sum_{j \neq q} \sum_{l \neq j} (A_r)_{pl}, \sum_{i \neq p} (A_r)_{iq} = \sum_{j \neq q} (A_r)_{pj} \Bigg\rangle \\
		&= \Bigg\langle (A_r)_{pq} \Biggm \vert (\alpha - 2)(J_r)_q = (\beta - 2)(I_r)_p = 0, \sum_{i \neq p} (A_r)_{iq} = \sum_{j \neq q} (A_r)_{pj} \Bigg\rangle \\
		&= \Bigg\langle (A_r)_{pq} \Biggm \vert (\alpha - 2)(J_r)_q = (\beta - 2)(I_r)_p = 0, (J_r)_q = (I_r)_p, \text{ for all }p,q \Bigg\rangle.
		\end{align*}
	We can rewrite each $(A_r)_{i1}$ and $(A_r)_{1j}$ in terms of the other $(A_r)_{ij}$, the $(J_r)_q$, and the $(I_r)_p$, and hence remove them from the list of generators. Then, since $(J_r)_q = (I_r)_p$ for all $1 \leq p \leq \alpha$, $1 \leq q \leq \beta$ we can remove all-but-one of these from the list of generators as well, leaving:
		\begin{multline}\label{group3}
		\coker = \langle (A_r)_{pq}, (J_r)_1 \mid (\alpha - 2)(J_r)_1 = (\beta - 2)(J_r)_1 = 0, \\
		\text{for }2 \leq p \leq \alpha, 2 \leq q \leq \beta, \text{ and }1 \leq r \leq t \rangle.
		\end{multline}
	We substitute $a = \alpha -2$, $b = \beta - 2$, and write $g := \gcd(a,b)$. Then (\ref{group3}) is a presentation for $\mathbb{Z}^{t(a+1)(b+1)} \oplus (\mathbb{Z}/g)^t$ if $g > 1$, and $\mathbb{Z}^{t(a+1)(b+1)}$ otherwise. If $\alpha = \beta = 2$, then (\ref{group3}) gives a presentation for $\mathbb{Z}^2$. Together with Theorem \ref{evans}, this gives the desired result.
\end{proof}

\begin{prop}\label{new_unpointed_polygon_identity}
	Let $\alpha, \beta \geq 3$, and let $\kappa = \kappa (\alpha, \beta)$ be the complete bipartite graph on $\alpha$ white and $\beta$ black vertices. Then for all $t \geq 1$, the order of the class of the identity $[\mathbf{1}]$ in $K_0(\cst(\Lambda_t^*(\kappa)))$ is equal to $g := \gcd(\alpha-2,\beta-2)$.
	
	Furthermore, the isomorphism class of $\cst(\Lambda_t^*(\kappa))$ is completely determined by the K-groups in Theorem \ref{new_polygons} and the order of $[\mathbf{1}]$ in $K_0$.\qed
\end{prop}

\section*{Acknowledgements}

The author wishes to acknowledge their advisor Alina Vdovina for introducing them to the subject and for her guidance throughout, and the referee for detailed and enlightening feedback. The author expresses their gratitude to Newcastle University for providing an excellent research environment, and to the EPSRC for funding this project.

\nocite{KonVdo2015}

\bibliographystyle{amsplain}
\bibliography{K-theory_bib}

\vspace{1cm}

\addresses

\end{document}